\documentclass[11pt]{article}

\usepackage[T1]{fontenc}
\usepackage{amsmath}
\usepackage{amssymb}
\usepackage{amscd}
\usepackage{amsbsy}
\usepackage{amsfonts}
\usepackage{amsthm}
\usepackage{bbm}
\usepackage{mathtools}
\usepackage{xcolor}
\usepackage{comment}

\usepackage{devanagari}

\input xy
\xyoption{all}

\newtheorem{thm}{Theorem}[section]
\newtheorem{lem}[thm]{Lemma}
\newtheorem{prop}[thm]{Proposition}
\newtheorem{cor}[thm]{Corollary}
\newtheorem{rem}[thm]{Remark}
\newtheorem{dfn}[thm]{Definition}

\newtheorem{ques}[thm]{Question}

\DeclareMathOperator{\SH}{\mathcal{SH}}
\DeclareMathOperator{\syn}{Syn}
\DeclareMathOperator{\Sp}{Sp}
\DeclareMathOperator{\sh}{Sh}

\DeclareMathOperator{\A}{\mathcal{A}}
\DeclareMathOperator{\G}{\mathcal{G}}
\DeclareMathOperator{\Pa}{\mathcal{P}}

\DeclareMathOperator{\mbp}{\mathrm{MBP}}
\DeclareMathOperator{\mgl}{\mathrm{MGL}}
\DeclareMathOperator{\MU}{\mathrm{MU}}
\DeclareMathOperator{\hz}{\mathrm{H}{\mathbb F}_2}
\DeclareMathOperator{\hzt}{\mathrm{H}{\mathbb F}_2^{top}}
\DeclareMathOperator{\F}{{\mathbb F}_2}
\DeclareMathOperator{\Z}{{\mathbb Z}}

\DeclareMathOperator{\R}{{\mathbb R}}
\DeclareMathOperator{\C}{{\mathbb C}}
\DeclareMathOperator{\sph}{{\mathbb S}}
\DeclareMathOperator{\rr}{Re_{\mathbb R}}
\DeclareMathOperator{\rc}{Re_{\mathbb C}}
\DeclareMathOperator{\un}{{\mathbbm 1}}

\DeclareMathOperator{\Fi}{{\mathbb F}}

\DeclareMathOperator{\gm}{{\mathbb G}_m}

\DeclareMathOperator{\km}{\mathrm{k}}
\DeclareMathOperator{\Mod}{\mathrm{Mod}}

\DeclareMathOperator{\Com}{\mathrm{Comod}}

\DeclareMathOperator{\spec}{\mathrm{Spec}}
\DeclareMathOperator{\ext}{\mathrm{Ext}}
\DeclareMathOperator{\cof}{\mathrm{cofib}}
\DeclareMathOperator{\fib}{\mathrm{fib}}
\DeclareMathOperator{\ce}{\mathrm{\check{C}}}
\DeclareMathOperator{\colim}{\mathrm{colim}}

\DeclareMathOperator{\coker}{\mathrm{coker}}
\DeclareMathOperator{\ind}{\mathrm{Ind}}
\DeclareMathOperator{\Hom}{\mathrm{Hom}}
\DeclareMathOperator{\tot}{\mathrm{Tot}}

\DeclareMathOperator{\ta}{\text{\footnotesize {\dn t}}}

\title{\textsc{Real isotropic motivic spectra}}
\author{Fabio Tanania}
\date{}

\begin{document}

\maketitle

 \begin{abstract}\let\thefootnote\relax\footnote{The author acknowledges support by the European Research Council (ERC) under Horizon Europe (grant agreement n° 101040935), by the Deutsche Forschungsgemeinschaft (DFG, German Research Foundation) TRR 326 \textit{Geometry and Arithmetic of Uniformized Structures}, project number 444845124 and the LOEWE professorship in Algebra, project number LOEWE/4b//519/05/01.002(0004)/87.}
 In this paper, we introduce the category of real isotropic motivic spectra, and show that the real realization functor from motivic spectra over $\R$ to classical spectra factors through it. We then describe its cellular subcategory as a one-parameter deformation of the category of spectra, with parameter $\rho$ corresponding to $-1 \in \R^{\times}$, whose special fiber is the derived category of comodules over the dual Steenrod algebra. This leads to an identification of real isotropic cellular spectra with $\F$-synthetic spectra, and sheds light on the relation between motivic homotopy theory over $\R$ and $\F$-synthetic homotopy theory.
 	\end{abstract}
 
\section{Introduction}

Isotropic motives were introduced by Vishik in \cite{V}. They are constructed from the Voevodsky's triangulated category of motives \cite{Vo} by annihilating the motives of anisotropic varieties. By now it is well known that the category of isotropic motives behaves particularly well over the so-called flexible fields, that is, purely transcendental extensions of infinite degree over some other field. Over such fields, there is plenty of anisotropic varieties, and by annihilating their motives the resulting categories are, in a certain sense, purified of the arithmetic information of the base field. For example, in \cite{V} we can find the computation of the isotropic motivic cohomology mod 2 of a flexible field. This is shown to be an exterior algebra over $\F$ with countably infinite generators that are independent of the base field. Moreover, in \cite{V1} isotropic Chow groups are identified with Chow groups modulo numerical equivalence (with $\Fi_p$-coefficients), showing that isotropic motives encode information about varieties which are closer to topology.

In \cite{T1} and \cite{T}, the author defined the isotropic motivic homotopy category, and showed that, over a flexible field, its cellular subcategory is purely algebraic, being identified with the derived category of comodules over the dual Steenrod algebra. In particular, the isotropic homotopy groups of the sphere spectrum, suitably completed, are the Ext-groups of the Steenrod algebra, namely the $E_2$-page of the Adams spectral sequence. Once again, this description is completely independent of the base field, provided that it is flexible.

This situation is reminiscent of what happens for the cellular subcategory of $p$-complete motivic spectra over the complex numbers. In fact, there is a map $\tau:\Sigma^{0,-1}\un \rightarrow \un$ in $\SH_{cell}(\C)_p$ such that the cofiber of $\tau$, denoted by $C\tau$, is a motivic $E_{\infty}$-ring spectrum \cite{G}, and $C\tau{\text-}\Mod_{cell}$ is identified with the derived category of even $\mathrm{BP}_*\mathrm{BP}$-comodules \cite{GWX}. On the other hand, by inverting $\tau$ in $\SH_{cell}(\C)_p$, one obtains the category of $p$-complete spectra \cite{DI2}.  Using a different approach, Pstr\k{a}gowski introduced the notion of synthetic spectra, and proved that $p$-complete cellular motivic $\C$-spectra are the same as $p$-complete even $\MU$-synthetic spectra \cite{P}. Altogether, these results tell us that $\SH_{cell}(\C)_p$ is a one-parameter deformation of $\SH_p$ with parameter $\tau$ and special fiber given by derived even comodules over $\mathrm{BP}_*\mathrm{BP}$. In particular, this description does not require any algebro-geometric input, and has very powerful applications to the computation of classical homotopy groups of spheres \cite{IWX}.

In \cite{P}, $E$-synthetic homotopy theory is developed with respect to any sufficiently good homotopy commutative ring spectrum $E$. As we have just mentioned, $\mathrm{BP}$-synthetic homotopy theory appears as the cellular subcategory of motivic spectra over $\C$. This intriguing phenomenon leads us to wonder whether other categories of synthetic spectra arise from motivic homotopy theory. In particular, due to its computational advantages, it would be especially interesting to understand the case of $\F$-synthetic homotopy theory, whose connection to motivic homotopy theory has so far remained rather elusive. This motivates the following question.

\begin{ques}\label{mainques}
	\normalfont
Can $\F$-synthetic spectra be realized motivically?
\end{ques}

The isotropic motivic homotopy category over  flexible fields seems to point in the right direction, although the parallelism ends with the description of the category of isotropic cellular spectra as derived comodules over the dual Steenrod algebra. More precisely, in its current form it is not possible to identify the isotropic sphere spectrum as the cofiber of some parameter in the category of cellular isotropic spectra. To address this issue, in this paper we propose a different construction of the category of isotropic spectra, one that is more suitable for studying motivic spectra over the real numbers (or more generally real closed fields). We build the category of real isotropic motivic spectra by annihilating in $\SH(\R)$ only the suspension spectra of the projective quadrics defined by the equation $\sum_{i=0}^n x_i^2=0$, rather than all anisotropic varieties as in \cite{T1}. Within this category we consider the parameter $\rho:\Sigma^{-1,-1}\un \rightarrow \un$ corresponding to $-1 \in \R^{\times}$. We then prove that inverting $\rho$ in the cellular subcategory of real isotropic motivic spectra yields the classical category of spectra, whereas annihilating $\rho$ recovers the derived category of comodules over the dual Steenrod algebra. Ultimately, this description allows us to identify real isotropic cellular spectra with $\F$-synthetic spectra.\\

\textbf{Main results.} We work in the category of real motivic spectra $\SH(\R)$, although all the results in this paper hold in $\SH(k)$ for any real closed field $k$.

Let $\un$ denote the unit in $\SH(\R)$, let $\rho:\Sigma^{-1,-1}\un \rightarrow \un$ be the map associeted with $-1 \in \R^{\times}$, and define the real isotropic sphere spectrum by:
$$\un^{iso} \coloneqq \colim_n \cof(\Sigma^{\infty}_{+}\ce(Q_{\rho^n}) \rightarrow \un),$$
where $Q_{\rho^n}$ is the Pfister quadric associated with the pure symbol $\rho^n$ and $\ce(Q_{\rho^n})$ is its \v{C}ech nerve.

We prove that $\un^{iso}$ is an $E_{\infty}$-algebra in $\SH(\R)$, which allows us to define the category of real isotropic motivic spectra by:
$$\SH^{iso}(\R) \coloneqq \un^{iso}{\text -}\Mod$$
equipped with a localization functor $L^{iso}:\SH(\R) \rightarrow \SH^{iso}(\R)$ that annihilates the suspension spectra of anisotropic Pfister quadrics.

We point out that, since $\R$ is not a flexible field, the category of real isotropic motivic spectra introduced in this paper is drastically different from the one studied in \cite{T1} over flexible fields. In particular, we will show that the cellular subcategory of $\SH^{iso}(\R)$, denoted by $\SH^{iso}_{cell}(\R)$, is not algebraic, and thus considerably more complicated than its flexible counterpart. Nevertheless, within $\SH^{iso}_{cell}(\R)$ we can still consider the parameter $\rho$ and study the special and generic fibers of the associated deformation.

Regarding the special fiber, we follow the approach of Gheorghe-Wang-Xu \cite{GWX}. More precisely, we first prove that $\un^{iso}/\rho$ is an $E_{\infty}$-algebra in $\SH(\R)$, where
$$\un^{iso}/\rho\coloneqq \cof(\Sigma^{-1,-1}\un^{iso}\xrightarrow{\rho}\un^{iso}).$$

Then, we proceed by studying the category of $\mbp^{iso}/\rho$-cellular modules, where $\mbp$ is the motivic Brown-Peterson spectrum at the prime 2, $\mbp^{iso}$ is its isotropic localization and $\mbp^{iso}/\rho \coloneqq \un^{iso} /\rho\wedge \mbp$. In the end, what we get is an identification of $\mbp^{iso}/\rho$-cellular modules with bigraded $\F$-vector spaces. This is the key step in constructing an isotropic Adams-Novikov spectral sequence with respect to $\mbp^{iso}$-homology, which serves as the main tool for proving the following result.

\begin{thm}[Theorem \ref{specfib}]
	There is an equivalence of stable $\infty$-categories:
	$$\un^{iso}/\rho{\text -}\Mod_{cell} \simeq {\mathcal D}(\A_*{\text -}\Com_*),$$
	where $\A_*$ is the dual Steenrod algebra.
\end{thm}

Once the special fiber of the deformation induced by $\rho$ has been described, we proceed to investigate the generic fiber. In particular, we show that the real realization functor $\rr:\SH(\R) \rightarrow \SH$ factors through $\SH^{iso}(\R)$. This means that real isotropic motivic homotopy theory interpolates between real motivic homotopy theory and classical homotopy theory. As a consequence, we obtain new cohomology theories for real algebraic varieties that are, in principle, more computable than the motivic ones, yet more informative than the classical topological ones. Moreover, by applying results of Bachmann \cite{B}, we arrive at the following conclusion.

\begin{thm}[Theorem \ref{equniso}]
	There is an equivalence of stable $\infty$-categories:
	$$\un^{iso}[\rho^{-1}]{\text -}\Mod \simeq \SH.$$
\end{thm}

By combining these two theorems, we obtain a description of the category of real isotropic cellular spectra as a one-parameter deformation of the category of spectra with parameter $\rho$ and special fiber given by derived comodules over the dual Steenrod algebra. Formally, this description coincides with that of $\F$-synthetic spectra. In fact, we can prove the following.

\begin{thm}[Proposition \ref{homgr} and Theorem \ref{upsi}]
	There exists a functor:
	$$\Theta^*:\SH^{iso}_{cell}(\R)\rightarrow \syn_{\F}$$ 
	that is an equivalence of stable $\infty$-categories, sending $\rho$ to the parameter in $\F$-synthetic spectra (up to a sign).
\end{thm}

This result answers Question \ref{mainques} affirmatively, by showing that, just as $p$-complete even $\MU$-synthetic spectra arise in motivic homotopy theory as $p$-complete cellular motivic spectra over $\C$, $\F$-synthetic spectra arise as isotropic motivic cellular spectra over $\R$, establishing a clear connection between $\F$-synthetic homotopy theory and motivic homotopy theory over the real numbers, an aspect that was so far not completely understood. We also want to highight that the identification provided by the previous theorem does not require any $2$-completion, in constrast with the one in \cite{P}. The reason is that $\rho$ exists already in the non-complete setting, and $\un^{iso}/\rho$ is $2$-complete. 

The category of real Artin-Tate motivic spectra was studied in detail in \cite{BHS}. In particular, the category of $2$-complete real Artin-Tate motivic spectra is identified with a one-parameter deformation of the $C_2$-equivariant stable homotopy category with parameter $\ta$ and special fiber given by the derived category of Mackey-functor $\MU_*\MU$-comodules. Moreover, \cite{BHS} provides a functor from $2$-complete real Artin-Tate motivic spectra to $2$-complete $\F$-synthetic spectra. Our approach via real isotropic spectra shows that this functor already exists at the non-complete level and, furthermore, allows us to explicitly describe the induced functor on the special fibers, thereby answering a question of Burklund-Hahn-Senger. 

	\begin{thm}[Theorem \ref{tosyn} and Proposition \ref{last}]
	The isotropic localization functor $L^{iso}:\SH(\R) \rightarrow \SH^{iso}(\R)$ restricts to a functor $L^{iso}:\SH(\R)^{\mathrm{AT}} \rightarrow \SH^{iso}_{cell}(\R)$ that induces a morphism of Hopf algebroids:
	$$(\pi_{2*,*}(\mbp),\mbp_{2*,*}(\mbp)) \rightarrow (\pi_{2*,*}(\mbp^{iso}),\mbp^{iso}_{2*,*}(\mbp^{iso}))$$
	that is the quotient map
	$$(\mathrm{BP}_*, \mathrm{BP}_*\mathrm{BP}) \rightarrow (\F, \A_*)$$
	sending $v_i$ to 0 and $t_i$ to $\xi_i$.
\end{thm}

\textbf{Outline.} In Section 2, we fix some notation that we use throughout this paper. In Section 3, we introduce the category of real isotropic motivic spectra and study the main properties of real isotropic motivic cohomology and real isotropic $\mbp$-cohomology. In Section 4, we continue by investigating the isotropic cofiber of $\rho$ and identifying $\mbp^{iso}/\rho$-cellular modules with bigraded $\F$-vector spaces. In Section 5, we completely solve the problem of describing the special fiber of the deformation induced by $\rho$ by identifying $\un^{iso}/\rho$-cellular modules with derived comodules over the dual Steenrod algebra. In Section 6, we proceed by studying the generic fiber, showing that inverting $\rho$ in real isotropic motivic spectra recovers the category of spectra. In Section 7, we construct an explicit functor from real isotropic cellular spectra to $\F$-synthetic spectra, and prove that it is an equivalence. In Section 8, we conclude by analyzing the relation with real Artin-Tate motivic spectra.\\

\textbf{Acknowledgements.} I would like to thank Dan Isaksen for suggesting that isotropic spectra might be related to $\F$-synthetic spectra, and Alexey Ananyevskiy for raising the question of how isotropic spectra behave over the real numbers.\\

\section{Notation}

\begin{tabular}{c|c}
	$\km^M_*(-)$ & Milnor K-theory mod 2\\
	$\mathrm{GW}(-)$ & Grothendieck-Witt ring\\
	$\SH(\R)$ & stable motivic homotopy category over $\R$\\
	$\SH(\R)^{\mathrm{AT}}$ & category of real Artin-Tate motivic spectra\\
	$\un$ & motivic sphere spectrum\\
	$\mgl$ & algebraic cobordism spectrum\\
	$\mbp$ & motivic Brown-Peterson spectrum at the prime 2\\
	$\hz$ & motivic Eilenberg-MacLane spectrum of $\F$\\
	$\pi_{**}(-)$& motivic homotopy groups\\
	$\A_{**}$ & motivic dual Steenrod Algebra\\
	$\SH^{iso}(\R)$ & real isotropic stable motivic homotopy category\\
	$\SH^{iso}_{cell}(\R)$ & category of real isotropic cellular spectra\\
	$\un^{iso}$ & real isotropic motivic sphere spectrum\\
	$\mbp^{iso}$ & isotropic motivic Brown-Peterson spectrum at the prime 2\\
	$\hz^{iso}$ & isotropic motivic Eilenberg-MacLane spectrum of $\F$\\
	$\A^{iso}_{**}$ & isotropic motivic dual Steenrod Algebra\\
	$\SH$ & stable homotopy category\\
	$\sph$ & sphere spectrum\\
	$\mathrm{BP}$ & Brown-Peterson spectrum at the prime 2\\
	$\hzt$ & Eilenberg-MacLane spectrum of $\F$\\
	$\pi_*(-)$& homotopy groups\\
	$\A_*$ & dual Steenrod algebra\\
	$\syn_{\F}$&category of $\F$-synthetic spectra\\
	$\sph^{0,0}$ & $\F$-synthetic sphere spectrum\\
	$\pi_{**}^{\syn}(-)$ & $\F$-synthetic homotopy groups\\
	$\rr$& real realization functor\\
	${\mathcal D}^b(-)$ & bounded derived category of an abelian category\\
	${\mathcal D}(-)$ & ind-completion of ${\mathcal D}^b(-)$
\end{tabular}\\

For any $E_{\infty}$-algebra $R$ in $\SH(\R)$, we denote by $R{\text -}\Mod$ the stable $\infty$-category of $R$-modules, by $- \wedge_R -$ its smash product, and by $[-,-]_{R}$ the hom-sets in the homotopy category of $R{\text -}\Mod$, namely
$$[-,-]_{R} \coloneqq \pi_0\mathrm{Map}_{R{\text -}\Mod}(-,-).$$

In the case where $R$ is the motivic sphere spectrum, we drop the decoration and simply write $- \wedge -$ for the smash product, and $[-,-]$ for the homotopy hom-sets.

For any algebra $A$ (resp. coalgebra $C$), we denote by $A{\text -}\Mod$ (resp. $C{\text -}\Com$) the abelian category of left $A$-modules (resp. $C$-comodules), and by $\Hom_A(-,-)$ (resp. $\Hom_C(-,-)$) its hom-sets. 

For any bigraded object $M_{**}$ in $A{\text -}\Mod$ (resp. $C{\text -}\Com$), we define its $(p,q)$-suspension by: $$\Sigma^{p,q}M_{a,b}\coloneqq M_{a-p,b-q}.$$
 
Finally, for any pair of bigraded objects $M_{**}$ and $N_{**}$, we define the bigraded hom-sets by:
$$\Hom ^{p,q}(M_{**},N_{**})\coloneqq \Hom ^{0,0}(\Sigma^{p,q}M_{**},N_{**}),$$
where $\Hom ^{0,0}(-,-)$ is the set of degree-preserving homomorphisms.  

\section{Real isotropic spectra}

In this section, we begin by constructing the isotropic motivic stable homotopy category over the real numbers. We then proceed to study real isotropic motivic cohomology and real isotropic $\mbp$-cohomology, and present the main computations that will be needed in the subsequent sections.

Denote by $\rho: S^0 \rightarrow \gm$ the map corresponding to $-1 \in \R^{\times}$. Let $Q_{\rho^n}$ be the Pfister quadric defined by the equation $\langle\langle -1 \rangle \rangle^{\otimes n}=0$, where $\langle \langle -1 \rangle \rangle \coloneqq \langle 1,1 \rangle$ is the one-fold Pfister form corresponding to the pure symbol $\rho \in \km^M_1(\R)$.

\begin{dfn}
	\normalfont
We define the real isotropic sphere spectrum as the object in $\SH(\R)$ given by:
$$\un^{iso} \coloneqq \colim_n \cof(\Sigma^{\infty}_{+}\ce(Q_{\rho^n}) \rightarrow \un),$$
where $\ce(Q_{\rho^n})$ is the \v{C}ech nerve of the map $Q_{\rho^n} \rightarrow \spec(\R)$.
\end{dfn}

\begin{rem}
\normalfont
We point out that the definition of the real isotropic sphere spectrum provided above is not equivalent to the one in \cite[Definition 3.2]{T}. In fact, over a flexible field $k$, that is, $k=k_0(t_1,t_2,\dots)$ for some field $k_0$, the isotropic sphere spectrum defined in \cite[Definition 3.2]{T} as
$$\un^{iso}_k \coloneqq \colim_P \cof(\Sigma^{\infty}_{+}\ce(P) \rightarrow \un_k),$$
where the colimit is taken over all anisotropic projective varieties $P$ over $k$, is equivalent to
$$\un^{iso}_k \simeq \colim_{\alpha} \cof(\Sigma^{\infty}_{+}\ce(Q_{\alpha}) \rightarrow \un_k),$$
where the colimit now is taken over all non-zero pure symbols $\alpha \in \km^M_*(k)$, and $Q_{\alpha}$ is the corresponding Pfister quadric. This equivalence holds because, over flexible fields, every anisotropic variety is contained in an anisotropic Pfister quadric.

Since $\R$ is not a flexible field, the two definitions yield different spectra over the real numbers. In particular, this means that the isotropic motivic homotopy category over $\R$, which is the main object of study in this paper, differs drastically from the one analyzed in \cite{T} and \cite{T1}.
\end{rem}

\begin{prop}\label{idm}
The motivic spectrum $\un^{iso}$ is an idempotent monoid in $\SH(\R)$. In particular, $\un^{iso}$ is an $E_{\infty}$-algebra in $\SH(\R)$.
\end{prop}
\begin{proof}
Since there are equivalences of motivic spaces $\ce(Q_{\rho^n}) \times \ce(Q_{\rho^n}) \simeq \ce(Q_{\rho^n} \times Q_{\rho^n}) \simeq \ce(Q_{\rho^n})$ for all $n$ (see for example \cite[Lemma 6.2]{T2}), we deduce that $\cof(\Sigma^{\infty}_{+}\ce(Q_{\rho^n}) \rightarrow \un)$ is an idempotent monoid in $\SH(\R)$ for all $n$.  Since the smash product commutes with colimits, we conclude that:
\begin{align*}
	\un^{iso}\wedge \un^{iso} &\simeq \colim_{m,n}(\cof(\Sigma^{\infty}_{+}\ce(Q_{\rho^m}) \rightarrow \un) \wedge \cof(\Sigma^{\infty}_{+}\ce(Q_{\rho^n}) \rightarrow \un)) \\
	&\simeq \colim_n (\cof(\Sigma^{\infty}_{+}\ce(Q_{\rho^n}) \rightarrow \un) \wedge \cof(\Sigma^{\infty}_{+}\ce(Q_{\rho^n}) \rightarrow \un)) \\
	&\simeq \colim_n \cof(\Sigma^{\infty}_{+}\ce(Q_{\rho^n}) \rightarrow \un) = \un^{iso}.
	\end{align*}

This implies that $\un^{iso}$ is an idempotent monoid, and so an $E_{\infty}$-algebra in $\SH(\R)$ by \cite[Proposition 4.8.2.9]{Lu}.
\end{proof}

\begin{rem}
\normalfont
Thanks to Proposition \ref{idm}, we can consider the symmetric monoidal stable $\infty$-category of modules over $\un^{iso}$, which we denote by $\un^{iso}{\text -}\Mod$. According to \cite[Proposition 4.8.2.10]{Lu}, this category can be identified with the full subcategory $\SH^{iso}(\R)$ of $\SH(\R)$, whose objects are of the form $\un^{iso} \wedge E$, for any spectrum $E$ in $\SH(\R)$. In fact, there is an adjunction:
$$L^{iso}: \SH(\R) \leftrightarrows \SH^{iso}(\R): \iota,$$
where $L^{iso}$ is the smashing localization functor defined on objects by $L^{iso}E \coloneqq \un^{iso} \wedge E$.
\end{rem}

\begin{dfn}
	\normalfont
Denote by $E^{iso}$ the object $L^{iso}E$ in $\SH^{iso}(\R)$, for all $E$ in $\SH(\R)$. We call $\SH^{iso}(\R)$ the category of real isotropic motivic spectra. 
\end{dfn}

\begin{rem}
\normalfont
The definition of the real isotropic sphere spectrum, and consequently the definition of $\SH^{iso}(\R)$, work more generally over all formally real fields. Moreover, although we state all results in this paper over $\R$, they in fact hold over any real closed field.
\end{rem}

\begin{rem}\label{annull}
	\normalfont
	The real isotropic localization functor $L^{iso}$ introduced above has the effect of ``annihilating" anisotropic quadrics. In fact, over $\R$, any anisotropic projective quadric $Q$ is contained in $Q_{\rho^m}$ for a sufficiently large $m$. This implies that the projection $\ce(Q_{\rho^n})\times Q \rightarrow Q$ is an equivalence of motivic spaces for all $n \geq m$, from which it follows that 
	$$\cof(\Sigma^{\infty}_+(\ce(Q_{\rho^n})\times Q) \rightarrow \Sigma^{\infty}_+Q)\simeq 0$$
	in $\SH(\R)$, for all $n\geq m$. As a consequence, we deduce that
	$$(\Sigma^{\infty}_+Q)^{iso}=L^{iso}(\Sigma^{\infty}_+Q)=\un^{iso}\wedge \Sigma^{\infty}_+Q\simeq0.$$
\end{rem}

\begin{prop}\label{pihiso}
There is an isomorphism of $\F$-algebras:
$$\pi_{**}(\hz^{iso}) \cong \frac{\F[\rho,r_i: i \geq 0]}{(r_i^2-\rho r_{i+1}: i \geq 0)},$$
with generators $\rho$ in degree $(-1,-1)$ and $r_i$ in degree $(2^{i+1}-1,2^i-1)$.
\end{prop}
\begin{proof}
	By definition, we have the following isomorphisms:
	$$\pi_{**}(\hz^{iso}) \cong \hz_{**}(\un^{iso}) \cong \colim_n \hz_{**}(\cof(\Sigma^{\infty}_{+}\ce(Q_{\rho^n}) \rightarrow \un)).$$
	From \cite[Theorem 3.5]{V}, we also know that
	$$\hz_{**}(\cof(\Sigma^{\infty}_{+}\ce(Q_{\rho^n}) \rightarrow \un)) \cong \frac{R_{\rho^n}[r_i:0\leq i \leq n-1]}{(r_i^2-\rho r_{i+1}:0\leq i \leq n-2)},$$
	where $R_{\rho^n} \cong \km^M_*(\R)/\ker(\cdot \rho^n)$ as $\F$-algebras. Since $\km^M_*(\R) \cong \F[\rho]$, we get an isomorphism $R_{\rho^n} \cong \F[\rho]$ for every $n \geq 0$. Taking the colimit, we then obtain the desired result.
	\end{proof}

\begin{lem}\label{ringhom}
	The map of ring spectra $\hz \rightarrow \hz^{iso}$ induces in homotopy groups the ring homomorphism: 
	$$\pi_{**}(\hz) \cong \F[\tau,\rho] \rightarrow \pi_{**}(\hz^{iso}) \cong \frac{\F[\rho,r_i: i \geq 0]}{(r_i^2-\rho r_{i+1}: i \geq 0)},$$
	which sends $\rho$ to $\rho$ and $\tau$ to $\rho r_0$. 
\end{lem}
\begin{proof}
	Since $\rho$ is a canonical element, it suffices to check where $\tau$ is mapped. By degree reasons, its image in $\pi_{0,-1}(\hz^{iso})= (\hz^{iso})^{0,1}$ is either $0$ or $\rho r_0$. Note that $\rho=Sq^1\tau$ in $\hz^{0,1}$, so the image of $\tau$ in $(\hz^{iso})^{0,1}$ cannot be zero. This completes the proof.
	\end{proof}

\begin{prop}
	There is an isomorphism of $\pi_{**}(\hz^{iso})$-algebras:
	$$\A^{iso}_{**}\coloneqq \hz^{iso}_{**}(\hz^{iso}) \cong \frac{\F[\rho,r_i,\tau_i,\xi_{i+1}: i \geq 0]}{(r_i^2-\rho r_{i+1},\tau_i^2-\rho\tau_{i+1}-\rho(r_0+\tau_0)\xi_{i+1}: i \geq 0)},$$
	with generators $\tau_i$ in degree $(2^{i+1}-1,2^i-1)$ and $\xi_{i+1}$ in degree $(2^{i+2}-2,2^{i+1}-1)$.
\end{prop}
\begin{proof}
	Since $\un^{iso}$ is an idempotent monoid in $\SH(\R)$, we obtain an isomorphism:
	$$\hz^{iso}_{**}(\hz^{iso}) = \pi_{**}(\hz^{iso} \wedge \hz^{iso}) \cong \pi_{**}(\hz^{iso} \wedge \hz).$$ 
	By \cite[Proposition 5.5]{HKO}, it follows from the fact that $\hz^{iso}$ is an $\hz$-module that 
	$$\pi_{**}(\hz^{iso} \wedge \hz) \cong \pi_{**}(\hz^{iso}) \otimes_{\pi_{**}(\hz)} \A_{**},$$
	where $\A_{**}\coloneqq \hz_{**}(\hz)$ denotes the motivic dual Steenrod algebra over $\R$.
	
	Recall from \cite[Theorem 5.6]{HKO} that
	$$\A_{**} \cong \frac{\F[\tau,\rho,\tau_i,\xi_{i+1}: i \geq 0]}{(\tau_i^2-\tau\xi_{i+1}-\rho\tau_{i+1}-\rho\tau_0\xi_{i+1}: i \geq 0)}.$$
	Therefore, by Proposition \ref{pihiso} and Lemma \ref{ringhom}, we deduce that
	$$\A^{iso}_{**} \cong \pi_{**}(\hz^{iso} \wedge \hz) \cong \frac{\F[\rho,r_i,\tau_i,\xi_{i+1}: i \geq 0]}{(r_i^2-\rho r_{i+1},\tau_i^2-\rho\tau_{i+1}-\rho(r_0+\tau_0)\xi_{i+1}: i \geq 0)},$$
	which concludes the argument.
	\end{proof}

We now recall from \cite{DI} the definition of $R$-cellular modules, together with a well-known result concerning their structure. This statement holds in $\SH(k)$ for any field $k$.

\begin{dfn}
\normalfont
Let $R$ be an $E_{\infty}$-algebra in $\SH(k)$. We denote by $R{\text -}\Mod_{cell}$ the full localizing subcategory of $R{\text -}\Mod$ generated by $\Sigma^{0,q}R$ for all $q \in \Z$.
\end{dfn}

\begin{lem}\label{modcell}
	Let $R$ be an $E_{\infty}$-algebra in $\SH(k)$, and let $M$ be an object in $R{\text -}\Mod_{cell}$ such that $\pi_{**}(M)$ is isomorphic to the free $\pi_{**}(R)$-module $\bigoplus_{\alpha \in A}\pi_{**}(R)\cdot x_{\alpha}$, with generators $x_{\alpha} \in \pi_{p_{\alpha},q_{\alpha}}(M)$. Then, there is an equivalence of $R$-modules:
	$$M \simeq \bigvee_{\alpha \in A}\Sigma^{p_{\alpha},q_{\alpha}}R.$$
\end{lem}
\begin{proof}
	Each generator $x_{\alpha}$ of $\pi_{**}(M)$ induces a map of $R$-modules $x_{\alpha}:\Sigma^{p_{\alpha},q_{\alpha}}R \rightarrow M$. All together these morphisms yield a map
	$$f:\bigvee_{\alpha \in A}\Sigma^{p_{\alpha},q_{\alpha}}R \rightarrow M$$
	which, by hypothesis, induces an isomorphism on homotopy groups. Since $M$ is $R$-cellular, it follows from \cite[Section 7.9]{DI} that $f$ is an equivalence, thereby concluding the proof.
\end{proof}

\begin{rem}
	\normalfont
	Recall from \cite[Theorem 2.10]{OO} that the localization of the algebraic cobordism spectrum $\mgl$ at the prime 2, denoted by $\mgl_{(2)}$, splits as a wedge sum of shifted and twisted copies of a motivic spectrum, called motivic Brown-Peterson spectrum (at the prime 2), and denoted by $\mbp$. Specifically, there is an equivalence of motivic spectra:
	$$\mgl_{(2)}\simeq \bigvee_{\alpha \in A}\Sigma^{2q_{\alpha},q_{\alpha}}\mbp.$$
	
	Note that $\mbp$ is a homotopy commutative ring spectrum, but not an $E_{\infty}$-ring spectrum, as established in \cite[Theorem 1.1.2]{La}. 
	\end{rem}

\begin{dfn}
	\normalfont
	Let $\G_{**}$ be the bigraded Hopf algebra defined by:
	\begin{align*}
		\G_{p,q} \cong
	\begin{cases}
		\A_q & \mathrm{if} \: p=2q\\
		0 & \mathrm{if} \: p\neq2q
	\end{cases},
\end{align*}
where $\A_*$ is the topological dual Steenrod algebra $\F[\xi_i:i \geq 1]$, with generators $\xi_i$ of degree $2^i-1$.
\end{dfn}

\begin{prop}\label{hmbpiso}
	There is an isomorphism of left $\A_{**}^{iso}$-comodule algebras:
	$$\hz^{iso}_{**}(\mbp^{iso}) \cong \pi_{**}(\hz^{iso}) \otimes_{\F} \G_{**}.$$
\end{prop}
\begin{proof}
	Since $\mbp$ is a cellular spectrum, it follows that $\hz \wedge \mbp$ is an $\hz$-cellular module. Moreover, we know that
	$$\pi_{**}(\hz \wedge \mbp)=\hz_{**}(\mbp) \cong \Pa_{**} = \F[\tau,\rho,\xi_i:i \geq 1],$$
	where the identification in the middle is an isomorphism of left $\A_{**}$-comodule algebras by \cite[Theorem 6.11 and Remark 6.20]{H}. From Lemma \ref{modcell}, we deduce that there is an equivalence 
	$$\hz \wedge \mbp \simeq \bigvee_{\alpha \in A}\Sigma^{p_{\alpha},q_{\alpha}}\hz,$$
	where $A$ is the set of monomials in the $\xi_i$'s. Therefore, we obtain the following sequence of isomorphisms:
	\begin{align*}
		\hz^{iso}_{**}(\mbp^{iso}) &= \pi_{**}(\hz^{iso} \wedge \mbp^{iso})\\
		&\cong \pi_{**}(\bigvee_{\alpha \in A}\Sigma^{p_{\alpha},q_{\alpha}}\hz^{iso}) \\
		&\cong \pi_{**}(\hz^{iso}) \otimes_{\pi_{**}(\hz)} \Pa_{**} \\
		&\cong \pi_{**}(\hz^{iso}) \otimes_{\F} \G_{**},
	\end{align*}
	which completes the proof.
\end{proof}

\begin{lem}\label{pimbppure}
	There is an isomorphism:
	$$\pi_{2*,*}(\mbp^{iso}) \cong \F,$$
	where $\pi_{2*,*}(\mbp^{iso})$ is the pure subalgebra of $\pi_{**}(\mbp^{iso})$.
\end{lem}
\begin{proof}
	By the definition of the isotropic sphere spectrum, we have an isomorphism:
	$$\pi_{2*,*}(\mbp^{iso}) \cong \colim_n [\Sigma^{2*,*}\un,\cof(\Sigma^{\infty}_+\ce(Q_{\rho^n}) \rightarrow \un) \wedge \mbp].$$
	
	Now, observe that $\cof(\Sigma^{\infty}_+\ce(Q_{\rho^n})\rightarrow \un)$ is an extension of $\Sigma^{i,0}\Sigma^{\infty}_+Q_{\rho^n}^{\times i}$ for $i \geq 0$. For $i>0$, using \cite[Theorem B.1]{BKWX} and the fact that $\mbp$ is a direct summand of $\mgl_{(2)}$, we deduce that
	$$[\Sigma^{2*,*}\un,\Sigma^{i,0}\Sigma^{\infty}_+Q_{\rho^n}^{\times i}\wedge \mbp] \cong [\Sigma^{2*,*}\Sigma^{\infty}_+Q_{\rho^n}^{\times i},\Sigma^{i+2d_i,d_i}\mbp] \cong \mbp^{2(d_i-*)+i,d_i-*}(Q_{\rho^n}^{\times i})\cong 0,$$
	where $d_i$ is the dimension of $Q_{\rho^n}^{\times i}$. Therefore, we can simplify the original expression, yielding the identification:
	\begin{align*}
		\pi_{2*,*}(\mbp^{iso}) &\cong \colim_n [\Sigma^{2*,*}\un,\cof(\Sigma^{\infty}_+Q_{\rho^n} \rightarrow \un) \wedge \mbp]\\
		&\cong \colim_n \coker(\mbp_{2*,*}(Q_{\rho^n}) \rightarrow \mbp_{2*,*}),
	\end{align*}
	where $\mbp_{2*,*}=\pi_{2*,*}(\mbp) \cong \Z_{(2)}[v_i:i \geq1]$. From \cite[Example 7.7]{DV}, we know that the image of $\mbp_{2*,*}(Q_{\rho^n}) \rightarrow \mbp_{2*,*}$ is the ideal $(2,v_1,\dots,v_{n-1})$. Consequently, passing to the colimit, we obtain the desired isomorphism $\pi_{2*,*}(\mbp^{iso}) \cong \F$.
\end{proof}

\begin{prop}\label{freembp}
	There is an equivalence of $\mgl^{iso}$-modules:
	$$\hz^{iso} \simeq \bigvee_{I} \Sigma^{p_I,q_I}\mbp^{iso},$$
	which induces an isomorphism of $\pi_{**}(\mbp^{iso})$-modules 
	$$\pi_{**}(\hz^{iso})\cong \bigoplus_I\pi_{**}(\mbp^{iso})\cdot r_I,$$
	and an isomorphism of left $\mbp^{iso}_{**}(\mbp^{iso})$-comodules
	$$\mbp^{iso}_{**}(\hz^{iso})\cong \bigoplus_I\mbp^{iso}_{**}(\mbp^{iso})\cdot r_I,$$ 
	where $I$ runs over all possible sequences $\{i_1 < \dots < i_n\}$ of non-negative integers, $r_I\coloneqq r_{i_1}\cdots r_{i_n}$ and $(p_I,q_I)$ is the degree of $r_I$.
\end{prop}
\begin{proof}
	By \cite[Theorem 7.12]{H}, we know that there is an equivalence of $\mgl$-modules:
	$$\hz \simeq \mbp/(v_0,v_1,\dots),$$
	where $\mbp/(v_0,v_1,\dots) \coloneqq \colim_n\mbp/(v_0,v_1,\dots,v_n)$ and
	$$\mbp/(v_0,v_1,\dots,v_n) \coloneqq \cof(\Sigma^{2^{n+1}-2,2^n-1}\mbp/(v_0,v_1,\dots,v_{n-1}) \xrightarrow{v_n}\mbp/(v_0,v_1,\dots,v_{n-1})).$$
	
	It follows from Lemma \ref{pimbppure} that $v_n$ vanishes in $\pi_{**}(\mbp^{iso})$ for all $n \geq 0$. Therefore, applying the isotropic localization functor gives us equivalences of $\mgl^{iso}$-modules:
	$$\mbp^{iso}/(v_0,v_1,\dots,v_n) \simeq \mbp^{iso}/(v_0,v_1,\dots,v_{n-1}) \vee \Sigma^{2^{n+1}-1,2^n-1}\mbp^{iso}/(v_0,v_1,\dots,v_{n-1}).$$
	
	By induction on $n$, after taking the colimit, we obtain an equivalence:
	$$\hz^{iso} \simeq \bigvee_{I} \Sigma^{p_I,q_I}\mbp^{iso},$$
	where $(p_I,q_I)$ is the degree of $r_I$.  Consequently, this equivalence induces on homotopy groups an isomorphism of $\pi_{**}(\mgl^{iso})$-modules, and so of $\pi_{**}(\mbp^{iso})$-modules since $\pi_{**}(\mgl^{iso})$ is a free $\pi_{**}(\mbp^{iso})$-module: 
	$$\pi_{**}(\hz^{iso})\cong \bigoplus_I\pi_{**}(\mbp^{iso})\cdot r_I.$$
	
	Similarly, after applying $\mbp^{iso}$-homology, we obtain the isomorphism of left $\mbp^{iso}_{**}(\mbp^{iso})$-comodules:
	$$\mbp^{iso}_{**}(\hz^{iso})\cong \bigoplus_I\mbp^{iso}_{**}(\mbp^{iso})\cdot r_I$$
	that is what we aimed to show. 
\end{proof}

\begin{thm}\label{pimbp}
	There is an isomorphism of Hopf algebras:
	$$(\pi_{**}(\mbp^{iso}),\mbp^{iso}_{**}(\mbp^{iso})) \cong (\F[\rho],\F[\rho]\otimes_{\F}\G_{**}).$$
\end{thm}
\begin{proof}
	We begin by observing that the $\pi_{**}(\mbp^{iso})$-submodule of $\pi_{**}(\hz^{iso})$ generated by the $r_I$'s for all non-empty $I$'s is the ideal of $\pi_{**}(\hz^{iso})$ generated by the $r_i$'s for all $i \geq 0$. In fact, let $x$ be a homogeneous element in $(r_I:I\neq \emptyset)$, then $x=\rho^nr_I$ for some $n \geq 0$ and non-empty $I$, by Proposition \ref{pihiso}. Now, note that $\rho^n=\sum_jr_{I_j}y_j$ for some $I_j$, and some $y_j$ in $\pi_{**}(\mbp^{iso})$, by Proposition \ref{freembp}. Thus, $x=\sum_jr_Ir_{I_j}y_j$ which belongs to the $\pi_{**}(\mbp^{iso})$-submodule of $\pi_{**}(\hz^{iso})$ generated by the $r_I$'s for all non-empty $I$'s.
	
	Hence, we can compute:
	$$\pi_{**}(\mbp^{iso}) \cong \frac{\pi_{**}(\hz^{iso})}{(r_i : i \geq 0)} \cong \F[\rho]$$
	by Proposition \ref{pihiso}.
	
	Similarly, we can compute:
	$$\mbp^{iso}_{**}(\mbp^{iso}) \cong \frac{\mbp^{iso}_{**}(\hz^{iso})}{(r_i : i \geq 0)} \cong \F[\rho]\otimes_{\F}\G_{**}$$
	by Proposition \ref{hmbpiso}, which concludes the proof.
\end{proof}

\begin{prop}\label{mbpring}
	The homotopy commutative ring structure on $\mbp^{iso}$ extends to an $E_{\infty}$-ring structure. In particular, $\mbp^{iso}$ is an $E_{\infty}$-algebra in $\SH(\R)$.
\end{prop}
\begin{proof}
	By \cite[Proposition 1.4.4.11]{Lu}, there is a $t$-structure on $\un^{iso}{\text -}\Mod$ with non-negative part generated by $\Sigma^{n,n}\un^{iso}$ for all $n \in \Z$. Note that $\mbp^{iso}$ belongs to the non-negative part of this $t$-structure, as $\mbp$ is cellular and connective. On the other hand, by Theorem \ref{pimbp}, we know that $\mbp^{iso}$ also belongs to the non-positive part. Therefore,  $\mbp^{iso}$ is a homotopy commutative ring spectrum in the heart of the $t$-structure considered above. Thus, it is an $E_{\infty}$-ring spectrum.
\end{proof}

The last proposition allows us to talk safely about the symmetric monoidal stable $\infty$-category of $\mbp^{iso}$-modules.
	
\section{The isotropic cofiber of $\rho$}

At this point, we are ready to introduce and study a new isotropic motivic spectrum, the isotropic cofiber of $\rho$, which will play a fundamental role in the following sections.

\begin{dfn}
	\normalfont
	Define the isotropic cofiber of $\rho$ as the motivic spectrum in $\SH(\R)$ given by:
	$$\un^{iso}/\rho \coloneqq \cof(\Sigma^{-1,-1}\un^{iso} \xrightarrow{\rho}\un^{iso}).$$
	
	For any spectrum $E$ in $\SH(\R)$, we denote by $E^{iso}/{\rho}$ the spectrum $\un^{iso}/{\rho} \wedge E$.
\end{dfn}

\begin{prop}\label{2comp}
	The isotropic spectrum $\un^{iso}/\rho$ is $2$-complete.
\end{prop}
\begin{proof}
	Note that $\spec(\C) \cong Q_{\rho}$ as $\R$-schemes. It follows that $(\Sigma^{\infty}_+\spec(\C))^{iso} \simeq 0$. Therefore, the Euler characteristic of $\spec(\C)$ in $\pi_{0,0}(\un) \cong \mathrm{GW}(\R)$, which is $\langle2\rangle(1+\langle-1\rangle)$ by \cite[Corollary 11.2]{Le}, vanishes in $\pi_{0,0}(\un^{iso})$. This leads to the relation
	$$2+\eta\rho=1+\langle-1\rangle=0$$
	in $\pi_{0,0}(\un^{iso})$, implying that 2 vanishes in $\pi_{0,0}(\un^{iso}/\rho)$. Hence, $\un^{iso}/\rho$ is 2-torsion, and consequently 2-complete.
	\end{proof}

\begin{rem}\label{orcomp}
	\normalfont
	If $E$ is an $\mgl$-module, then $E^{iso}$ is both 2-complete and $\eta$-complete. In fact, the relation $\eta=0$ in $\pi_{**}(E^{iso})$ also implies that $2=0$ since, as shown in the proof of Proposition \ref{2comp}, we have $2+\eta\rho=0$ in $\pi_{0,0}(\un^{iso})$. In particular, if $E^{iso}$ is connective, we conclude that $E^{iso}$ is $\hz$-complete by \cite[Theorem 2.3]{BO}.
\end{rem}

\begin{lem}\label{compisocomp}
	For any object $Y$ in $\SH^{iso}(\R)$ and any homotopy commutative ring spectrum $E$ in $\SH(\R)$ there is an equivalence of isotropic spectra:
$$Y^{\wedge}_{E}\simeq Y^{\wedge}_{E^{iso}}.$$
\end{lem}
\begin{proof}
	Let $E^{\bullet}$ be the \v{C}ech conerve of the unit map $\un \rightarrow E$. Then, the $E$-nilpotent completion of $Y$ is given by:
	$$Y^{\wedge}_E \simeq \tot (E^{\bullet} \wedge Y).$$
	
	Since $\SH^{iso}(\R)$ is closed under limits in $\SH(\R)$, it follows that $Y^{\wedge}_E$ is an isotropic spectrum for any homotopy commutative ring spectrum $E$.
	
	By hypothesis, we know that $Y \simeq \un^{iso} \wedge Y$. Hence, we obtain an equivalence of cosimplicial spectra:
	$$E^{\bullet} \wedge Y \simeq (E^{iso})^{\bullet} \wedge Y,$$
	which induces the following equivalence in $\SH^{iso}(\R)$:
	$$Y^{\wedge}_{E}\simeq \tot (E^{\bullet} \wedge Y) \simeq \tot ((E^{iso})^{\bullet} \wedge Y) \simeq Y^{\wedge}_{E^{iso}},$$
	completing the proof.
	\end{proof}

\begin{prop}\label{compmbp}
	For any object $Y$ in $\SH^{iso}(\R)$ there is an equivalence of isotropic spectra:
	$$Y^{\wedge}_{\mbp}\simeq Y^{\wedge}_{\hz}.$$
\end{prop} 
\begin{proof}
	By Lemma \ref{compisocomp}, it suffices to show that $Y^{\wedge}_{\mbp^{iso}}\simeq Y^{\wedge}_{\hz^{iso}}$. To this end, let us consider the bicompletion $Y^{\wedge}_{(\mbp^{iso},\hz^{iso})}$. 
	
	By Remark \ref{orcomp}, the isotropic spectrum $\mbp^{iso}$ is $\hz^{iso}$-complete. Since $(\mbp^{iso})^{\wedge n}\wedge Y$ is an $\mbp^{iso}$-module, and hence $\hz^{iso}$-complete for all $n \geq 1$, we deduce the equivalence:
	$$Y^{\wedge}_{(\mbp^{iso},\hz^{iso})} \simeq \tot (((\mbp^{iso})^{\bullet}\wedge Y)^{\wedge}_{\hz^{iso}}) \simeq \tot ((\mbp^{iso})^{\bullet}\wedge Y)\simeq Y^{\wedge}_{\mbp^{iso}}.$$
	
	On the other hand, since $\hz^{iso}$ is an $\mbp^{iso}$-module, it is $\mbp^{iso}$-complete. This implies that $(\hz^{iso})^{\wedge n}\wedge Y$ is $\mbp^{iso}$-complete for all $n \geq 1$, yielding the equivalence:
	$$Y^{\wedge}_{(\mbp^{iso},\hz^{iso})} \simeq \tot (((\hz^{iso})^{\bullet}\wedge Y)^{\wedge}_{\mbp^{iso}}) \simeq \tot ((\hz^{iso})^{\bullet}\wedge Y)\simeq Y^{\wedge}_{\hz^{iso}},$$
	which concludes the argument.
	\end{proof}

\begin{prop}\label{2etacomp}
	The isotropic spectrum $(\un^{iso})^{\wedge}_2$ is $\eta$-complete.
\end{prop}
\begin{proof}
	Recall that the category of very effective spectra is the subcategory of $\SH(\R)$ generated under colimits and extensions by $\Sigma^{\infty}_+X$ for all smooth $\R$-schemes $X$. Since the spectrum $\Sigma^{\infty}_+\ce(Q_{\rho^n})$ is very effective for every $n$, it follows that also $\un^{iso}$ is very effective. Therefore, we conclude that $(\un^{iso})^{\wedge}_2$ is $\eta$-complete by \cite[Theorem 5.1]{BH}.
\end{proof}

\begin{cor}\label{mbpcomp}
	The isotropic spectrum $\un^{iso}/\rho$ is $\mbp$-complete.
\end{cor}
\begin{proof}
It follows from Propositions \ref{2comp} and \ref{2etacomp} that
$$\un^{iso}/\rho \simeq (\un^{iso}/\rho)^{\wedge}_2 \simeq (\un^{iso}/\rho)^{\wedge}_{2,\eta}.$$

Since $\un^{iso}/\rho$ is a connective spectrum, by Proposition \ref{compmbp}, we obtain the following equivalences of spectra:
$$(\un^{iso}/\rho)^{\wedge}_{\mbp}\simeq(\un^{iso}/\rho)^{\wedge}_{\hz}\simeq (\un^{iso}/\rho)^{\wedge}_{2,\eta} \simeq \un^{iso}/\rho.$$

This completes the proof.
\end{proof}

\begin{prop}\label{pimbpisorho}
	There is an isomorphism of Hopf algebras:
	$$(\pi_{**}(\mbp^{iso}/\rho),\mbp^{iso}_{**}(\mbp^{iso}/\rho)) \cong (\F,\G_{**}).$$
\end{prop}
\begin{proof}
	The cofiber sequence
	$$\Sigma^{-1,-1}\mbp^{iso} \xrightarrow{\rho}\mbp^{iso} \rightarrow \mbp/\rho$$
	induces a long exact sequence of homotopy groups
	$$\dots \rightarrow \pi_{p+1,q+1}(\mbp^{iso})\xrightarrow{\cdot \rho} \pi_{p,q}(\mbp^{iso}) \rightarrow \pi_{p,q}(\mbp^{iso}/\rho) \rightarrow \pi_{p,q+1}(\mbp^{iso}) \xrightarrow{\cdot\rho} \dots.$$
	
	By Theorem \ref{pimbp}, the multiplication by $\rho$ is injective in $\pi_{**}(\mbp^{iso})\cong \F[\rho]$. Therefore, we have the isomorphism:
	$$\pi_{**}(\mbp^{iso}/\rho)\cong \frac{\pi_{**}(\mbp^{iso})}{(\rho)}\cong \F.$$
	
	A similar argument applied to the long exact sequence of $\mbp^{iso}$-homology groups gives the isomorphism:
	$$\mbp^{iso}_{**}(\mbp^{iso}/\rho)\cong \frac{\mbp^{iso}_{**}(\mbp^{iso})}{(\rho)}\cong \G_{**}.$$
	\end{proof}

\begin{prop}\label{mbpisoy}
	Let $Y$ be an object in $\SH^{iso}(\R)$. Then, there is an isomorphism of left $\mbp^{iso}_{**}(\mbp^{iso})$-comodules:
	$$\mbp^{iso}_{**}(\mbp^{iso} \wedge Y) \cong \mbp^{iso}_{**}(\mbp^{iso}) \otimes_{\pi_{**}(\mbp^{iso})}\mbp^{iso}_{**}(Y).$$
\end{prop}
\begin{proof}
	Since $\mbp$ is a cellular spectrum, it follows that $\mbp^{iso} \wedge \mbp^{iso} \simeq \mbp^{iso} \wedge \mbp$ is an object in $\mbp^{iso}{\text-}\Mod_{cell}$. By Theorem \ref{pimbp}, we know that $\mbp^{iso}_{**}(\mbp^{iso})=\pi_{**}(\mbp^{iso} \wedge \mbp^{iso})$ is a free $\pi_{**}(\mbp^{iso})$-module. Hence, Lemma \ref{modcell} implies that there is an equivalence:
	$$\mbp^{iso}\wedge \mbp^{iso} \simeq \bigvee_{\alpha \in A}\Sigma^{p_{\alpha},q_{\alpha}}\mbp^{iso},$$
	and so, by \cite[Lemma 5.4]{HKO}, we obtain the following sequence of isomorphisms:
	\begin{align*}
		\mbp^{iso}_{**}(\mbp^{iso} \wedge Y) &=\pi_{**}(\mbp^{iso}\wedge \mbp^{iso} \wedge Y)\\
		&\cong \pi_{**}((\mbp^{iso}\wedge \mbp^{iso}) \wedge_{\mbp^{iso}} (\mbp^{iso}\wedge Y))\\
		&\cong \pi_{**}(\mbp^{iso}\wedge \mbp^{iso}) \otimes_{\pi_{**}(\mbp^{iso})} \pi_{**}(\mbp^{iso}\wedge Y)\\
		&= \mbp^{iso}_{**}(\mbp^{iso}) \otimes_{\pi_{**}(\mbp^{iso})}\mbp^{iso}_{**}(Y),
	\end{align*}
	which completes the proof.
\end{proof}

\begin{cor}\label{mbpisorhoy}
	Let $Y$ be an object in $\SH^{iso}(\R)$ such that $\mbp^{iso}_{**}(Y)$ is annihilated by $\rho$. Then, there is an isomorphism of left $\G_{**}$-comodules:
	$$\mbp^{iso}_{**}(\mbp^{iso} \wedge Y) \cong \G_{**} \otimes_{\F}\mbp^{iso}_{**}(Y).$$
\end{cor}
\begin{proof}
	This follows immediatetely from Theorem \ref{pimbp}, Proposition \ref{mbpisoy} and the fact that $\rho$ annihilates $\mbp^{iso}_{**}(Y)$.
\end{proof}

Let $\overline{ \mbp^{iso}}\coloneqq \fib(\un^{iso}\rightarrow \mbp^{iso})$, and denote by $\overline {\G_{**}}$ the cokernel of $\F\rightarrow \G_{**}$.

\begin{lem}\label{overmbp}
	Let $Y$ be an object in $\SH^{iso}(\R)$ such that $\mbp^{iso}_{**}(Y)$ is annihilated by $\rho$. Then, for all $n \geq 0$, there are isomorphisms of left $\G_{**}$-comodules:
	$$\mbp^{iso}_{**}(\overline{\mbp^{iso}}^{\wedge n}\wedge Y) \cong \Sigma^{-n,0}\overline {\G_{**}} ^{\otimes n} \otimes_{\F} \mbp^{iso}_{**}(Y)$$
	and
	$$\mbp^{iso}_{**}(\mbp^{iso}\wedge\overline{\mbp^{iso}}^{\wedge n}\wedge Y) \cong \Sigma^{-n,0}\G_{**}\otimes_{\F}\overline {\G_{**}} ^{\otimes n} \otimes_{\F} \mbp^{iso}_{**}(Y).$$
\end{lem}
\begin{proof}
	We proceed by induction on $n$. For $n=0$, the first isomorphism is tautological, while the second follows directly from Corollary \ref{mbpisorhoy}. 
	
	Now, assume that both isomorphisms hold for $n-1$. Then, the cofiber sequence
	$$\overline{\mbp^{iso}}^{\wedge n}\wedge Y \rightarrow \overline{\mbp^{iso}}^{\wedge n-1}\wedge Y \rightarrow \mbp^{iso} \wedge \overline{\mbp^{iso}}^{\wedge n-1}\wedge Y$$
	induces a short exact sequence of $\mbp^{iso}$-homology groups:
	$$0 \rightarrow \Sigma^{-n+1,0}\overline {\G_{**}} ^{\otimes n-1} \otimes_{\F} \mbp^{iso}_{**}(Y)\rightarrow \Sigma^{-n+1,0}\G_{**}\otimes_{\F}\overline {\G_{**}} ^{\otimes n-1} \otimes_{\F} \mbp^{iso}_{**}(Y)$$
	$$\rightarrow \mbp^{iso}_{**}(\Sigma^{1,0}\overline{\mbp^{iso}}^{\wedge n}\wedge Y) \rightarrow 0.$$
	
	From this, we conclude that:
	$$\mbp^{iso}_{**}(\overline{\mbp^{iso}}^{\wedge n}\wedge Y) \cong \Sigma^{-n,0}\overline {\G_{**}} ^{\otimes n} \otimes_{\F} \mbp^{iso}_{**}(Y).$$
	
	Once the first isomorphism for $n$ is established, the second follows again by Corollary \ref{mbpisorhoy}.
\end{proof}

\begin{thm}\label{mustlab}
	For all $p,q \in \Z$, there is an isomorphism:
	$$\pi_{p,q}(\un^{iso}/\rho) \cong \ext^{2q-p,q}_{\A_*}(\F,\F).$$
\end{thm}
\begin{proof}
	To prove this, we use the motivic Adams-Novikov spectral sequence associated with the motivic ring spectrum $\mbp^{iso}$. For this purpose, consider the Postnikov system in $\SH^{iso}(\R)$:
$$
\xymatrix{
	\dots \ar@{->}[r] &  \overline{\mbp^{iso}}^{\wedge s} \wedge \un^{iso}/\rho \ar@{->}[r] \ar@{->}[d] &  \dots \ar@{->}[r]   & \overline{\mbp^{iso}} \wedge \un^{iso}/\rho  \ar@{->}[r] \ar@{->}[d]	 & \un^{iso}/\rho  \ar@{->}[d]  \\
	&	\mbp^{iso}\wedge \overline{ \mbp^{iso}}^{\wedge s} \wedge \un^{iso}/\rho  \ar@{->}[ul]^{\Sigma} & & \mbp^{iso} \wedge \overline{\mbp^{iso}} \wedge \un^{iso}/\rho  \ar@{->}[ul]^{\Sigma}  &	\mbp^{iso} \wedge \un^{iso}/\rho  \ar@{->}[ul]^{\Sigma}
.}
$$

By applying the functor $\pi_{**}$, we obtain a spectral sequence whose $E_1$-page is given by:
$$E_1^{s,t,u} \cong \pi_{t-s,u}(\mbp^{iso}\wedge (\overline{\mbp^{iso}})^{\wedge s} \wedge \un^{iso}/\rho)=\mbp^{iso}_{t-s,u}((\overline{\mbp^{iso}})^{\wedge s} \wedge \un^{iso}/\rho),$$
with the first differential given by:
$$d_1^{s,t,u}:E_1^{s,t,u} \rightarrow E_1^{s+1,t,u}.$$

Thus, we have the isomorphism:
$$E_1^{s,t,u} \cong \Hom ^{t,u}_{\G_{**}}(\F,\G_{**} \otimes_{\F} \overline{\G_{**}}^{\otimes s} \otimes_{\F} \mbp^{iso}_{**}(\un^{iso}/\rho)) \cong \Hom ^{t,u}_{\G_{**}}(\F,\G_{**} \otimes_{\F} \overline{\G_{**}}^{\otimes s}),$$
from which we deduce the standard description of the $E_2$-page:
$$E_2^{s,t,u} \cong \ext^{s,t,u}_{\G_{**}}(\F,\mbp^{iso}_{**}(\un^{iso}/\rho))\cong \ext^{s,t,u}_{\G_{**}}(\F,\F).$$

The spectral sequence we considered conditionally converges to the groups $\pi_{t-s,u}((\un^{iso}/\rho)^{\wedge}_{\mbp}) \cong \pi_{t-s,u}(\un^{iso}/\rho)$, where the isomorphism follows from Corollary \ref{mbpcomp}.

Since $E_1^{s,t,u} \cong 0$ for all $t \neq 2u$, we deduce that $E_r^{s,t,u} \cong 0$ for all $t \neq 2u$ and $r \geq 1$. Therefore, all differentials $d_r^{s,t,u}: E_r^{s,t,u} \rightarrow E_r^{s+r,t+r-1,u}$ vanish for $r \geq 2$. This implies that the spectral sequence collapses at the second page, and we obtain the isomorphism:
$$\pi_{p,q}(\un^{iso}/\rho) \cong \ext^{2q-p,2q,q}_{\G_{**}}(\F,\F) \cong  \ext^{2q-p,q}_{\A_*}(\F,\F),$$
which completes the proof.
	\end{proof}

\begin{lem}\label{van}
	We have the following isomorphisms:
	\begin{enumerate}
		\item $[\Sigma^{p,q}\un^{iso}/\rho,\un^{iso}/\rho] \cong 0 $ for all $p \in \Z$ and $q<0$,
	\item $[\Sigma^{p,0}\un^{iso}/\rho,\un^{iso}/\rho] \cong 0 $ for all $p \neq0 $,
	\item $[\un^{iso}/\rho,\un^{iso}/\rho] \cong \F $.
	\end{enumerate}
\end{lem}
\begin{proof}
	The cofiber sequence
	$$\Sigma^{-1,-1}\un^{iso} \xrightarrow{\rho} \un^{iso} \rightarrow \un^{iso}/\rho$$
	induces a long exact sequence 
	$$\dots \rightarrow [\Sigma^{p,q-1}\un^{iso},\un^{iso}/\rho] \rightarrow [\Sigma^{p,q}\un^{iso}/\rho,\un^{iso}/\rho] \rightarrow [\Sigma^{p,q}\un^{iso},\un^{iso}/\rho] \rightarrow \dots.$$
	
	Since $[\Sigma^{p,q}\un^{iso},\un^{iso}/\rho] \cong \pi_{p,q}(\un^{iso}/\rho) \cong 0$ when $q<0$, or when $q=0$ and $p \neq 0$, we immediately deduce the statements (1) and (2). For $p=q=0$, the long exact sequence above reduces to the isomorphism: 
	$$[\un^{iso}/\rho,\un^{iso}/\rho] \cong[\un^{iso},\un^{iso}/\rho]  \cong \pi_{0,0}(\un^{iso}/\rho) \cong\F$$
	that is what we wanted to prove.
	\end{proof}

\begin{lem}\label{smash}
	There is an equivalence of isotropic spectra:
	$$(\un^{iso}/\rho)^{\wedge n} \simeq \bigvee_{i=0}^{n-1} {n-1 \choose i} \Sigma^{0,-i}\un^{iso}/\rho.$$
\end{lem}
\begin{proof}
	We proceed by induction on $n$. For $n=2$, by smashing the cofiber sequence
	$$\Sigma^{-1,-1}\un^{iso} \xrightarrow{\rho} \un^{iso} \rightarrow \un^{iso}/\rho$$
	with $\un^{iso}/\rho$, we obtain the cofiber sequence
	$$\Sigma^{-1,-1}\un^{iso}/\rho \xrightarrow{\rho} \un^{iso}/\rho \rightarrow \un^{iso}/\rho \wedge \un^{iso}/\rho,$$
	where the first map is trivial by Lemma \ref{van}. Thus, we obtain an equivalence:
	$$\un^{iso}/\rho \wedge \un^{iso}/\rho \simeq \un^{iso}/\rho \vee \Sigma^{0,-1}\un^{iso}/\rho.$$
	
	Now, assume the equivalence holds for $n-1$. Then, we have the following sequence of equivalences:
	\begin{align*}
		(\un^{iso}/\rho)^{\wedge n} &\simeq (\un^{iso}/\rho)^{\wedge n-1} \wedge \un^{iso}/\rho\\
		&\simeq \bigvee_{i=0}^{n-2} {n-2 \choose i} \Sigma^{0,-i}\un^{iso}/\rho \wedge \un^{iso}/\rho\\
		&\simeq \bigvee_{i=0}^{n-2} {n-2 \choose i} \Sigma^{0,-i}(\un^{iso}/\rho \vee \Sigma^{0,-1}\un^{iso}/\rho)\\
		& \simeq \bigvee_{i=0}^{n-2} {n-2 \choose i} \Sigma^{0,-i}\un^{iso}/\rho \vee \bigvee_{i=1}^{n-1} {n-2 \choose i-1} \Sigma^{0,-i}\un^{iso}/\rho\\
		& \simeq \bigvee_{i=0}^{n-1} {n-1 \choose i} \Sigma^{0,-i}\un^{iso}/\rho.
		\end{align*}
	
	This completes the induction step.
	\end{proof}

\begin{prop}
	There is a unique multiplication $\mu:\un^{iso}/\rho \wedge \un^{iso}/\rho  \rightarrow \un^{iso}/\rho$ that is homotopy unital, homotopy associative and homotopy commutative.
\end{prop}
\begin{proof}
	By Lemma \ref{smash}, we know that $\un^{iso}/\rho \wedge \un^{iso}/\rho \simeq \un^{iso}/\rho \vee \Sigma^{0,-1}\un^{iso}/\rho$. Hence, it follows from Lemma \ref{van} that
	$$[\un^{iso}/\rho \wedge \un^{iso}/\rho,\un^{iso}/\rho] \cong [\un^{iso}/\rho,\un^{iso}/\rho] \oplus [\Sigma^{0,-1}\un^{iso}/\rho,\un^{iso}/\rho] \cong \F.$$
	
	Therefore, there is only one possible choice for a homotopy unital multiplication $\mu:\un^{iso}/\rho \wedge \un^{iso}/\rho  \rightarrow \un^{iso}/\rho$, namely the non-trivial map.
	
	For the same reason, if we denote by $\chi:\un^{iso}/\rho \wedge \un^{iso}/\rho \rightarrow \un^{iso}/\rho \wedge \un^{iso}/\rho$ the swap map, we then have $\mu \circ \chi = \mu$ in $[\un^{iso}/\rho \wedge \un^{iso}/\rho,\un^{iso}/\rho]$. This implies that $\mu$ is homotopy commutative.
	
	For homotopy associativity, observe that, again by Lemma \ref{smash}, we have an isomorphism:
	$$[\un^{iso}/\rho \wedge \un^{iso}/\rho \wedge \un^{iso}/\rho,\un^{iso}/\rho] \cong [\un^{iso}/\rho,\un^{iso}/\rho] \cong \F,$$
	which implies that the two non-trivial maps $\mu \circ (\mu \wedge 1)$ and $\mu \circ (1 \wedge \mu)$ in $[\un^{iso}/\rho \wedge \un^{iso}/\rho \wedge \un^{iso}/\rho,\un^{iso}/\rho]$ must coincide. This concludes the proof.
	\end{proof}

\begin{thm}\label{cofrhoring}
	The multiplication map $\mu$ on $\un^{iso}/\rho$ can be uniquely extended to an $E_{\infty}$-ring structure. In particular, $\un^{iso}/\rho$ is an $E_{\infty}$-algebra in $\SH(\R)$.
\end{thm}
\begin{proof}
	By \cite[Corollary 3.2]{G}, the obstructions to extending $\mu$ to an $E_{\infty}$-ring structure lie in the hom-groups:
	$$[\Sigma^{n-3}(\un^{iso}/\rho)^{\wedge m},\un^{iso}/\rho] \cong \bigoplus_{i=0}^{m-1} {m-1 \choose i} [\Sigma^{n-3,-i}\un^{iso}/\rho,\un^{iso}/\rho]$$
	for $n \geq 4$ and $2 \leq m \leq n$, where the isomorphism follows from Lemma \ref{smash}. Since $-i \leq 0$ and $n-3 \geq 1$, Lemma \ref{van} implies that $[\Sigma^{n-3,-i}\un^{iso}/\rho,\un^{iso}/\rho] \cong 0$.
	
	To prove the uniqueness of this $E_{\infty}$-ring structure, we have to check the vanishing of
	$$[\Sigma^{n-2}(\un^{iso}/\rho)^{\wedge m},\un^{iso}/\rho] \cong \bigoplus_{i=0}^{m-1} {m-1 \choose i} [\Sigma^{n-2,-i}\un^{iso}/\rho,\un^{iso}/\rho]$$
	for $n \geq 4$ and $2 \leq m \leq n$. These hom-groups are trivial by Lemma \ref{van}, completing the proof.
	\end{proof}

\begin{cor}\label{algmbp}
	The isotropic spectrum $\mbp^{iso}/\rho$ is an $E_{\infty}$-algebra in $\un^{iso}/\rho {\text -}\Mod$.
\end{cor}
\begin{proof}
It follows from Proposition \ref{mbpring}, Theorem \ref{cofrhoring}, and the equivalence
	$$\mbp^{iso}/\rho \simeq \mbp^{iso} \wedge \un^{iso}/\rho,$$
	 that $\mbp^{iso}/\rho$ is an $E_{\infty}$-algebra in $\un^{iso}/\rho {\text -} \Mod$. 
	\end{proof}

\begin{prop}\label{hommbp}
	Let $X$ and $Y$ be objects in $\mbp^{iso}/\rho{\text -}\Mod_{cell}$. Then, there is an isomorphism:
	$$[X,Y]_{\mbp^{iso}/\rho} \cong \Hom ^{0,0}_{\F}(\pi_{**}(X),\pi_{**}(Y)).$$
\end{prop}
\begin{proof}
By Proposition \ref{pimbpisorho}, we know that $\pi_{**}(\mbp^{iso}/\rho)\cong \F$, so both $\pi_{**}(X)$ and $\pi_{**}(Y)$ are free $\pi_{**}(\mbp^{iso}/\rho)$-modules. Hence, by Lemma \ref{modcell}, we have equivalences of $\mbp^{iso}/\rho$-modules: $X \simeq \bigvee_{\alpha \in A} \Sigma^{p_{\alpha},q_{\alpha}} \mbp^{iso}/\rho$ and $Y \simeq \bigvee_{\beta \in B} \Sigma^{p_{\beta},q_{\beta}} \mbp^{iso}/\rho$. Thus, we obtain isomorphisms:
	\begin{align*}
 	[X,Y]_{\mbp^{iso}/\rho}  &\cong [\bigvee_{\alpha \in A} \Sigma^{p_{\alpha},q_{\alpha}} \un,\bigvee_{\beta \in B} \Sigma^{p_{\beta},q_{\beta}} \mbp^{iso}/\rho]
 	\\
 	& \cong \prod_{\alpha \in A} \bigoplus_{\beta \in B} \pi_{p_{\alpha}-p_{\beta},q_{\alpha}-q_{\beta}}(\mbp^{iso}/\rho)\\
 	& \cong \prod_{\alpha \in A} \bigoplus_{\beta \in B} \Sigma^{p_{\beta}-p_{\alpha},q_{\beta}-q_{\alpha}}\F\\
 	& \cong \Hom^{0,0}_{\F}(\bigoplus_{\alpha \in A} \Sigma^{p_{\alpha},q_{\alpha}}\F,\bigoplus_{\beta \in B} \Sigma^{p_{\beta},q_{\beta}}\F)\\
 	& \cong \Hom^{0,0}_{\F}(\pi_{**}(X),\pi_{**}(Y)),	 
 \end{align*}
which concludes the proof.
\end{proof}
 
 \begin{dfn}
 	\normalfont
 	The Chow-Novikov degree of $\pi_{p,q}(-)$ is defined as the integer $p -2q$. 
 	
 	Denote by $\mbp^{iso}/\rho{\text -}\Mod^b_{cell}$ the full subcategory of bounded $\mbp^{iso}/\rho$-cellular modules, that is, objects in $\mbp^{iso}/\rho{\text -}\Mod_{cell}$ whose homotopy groups are non-trivial only for a finite number of Chow-Novikov degrees. We also define the following full subcategories of $\mbp^{iso}/\rho{\text -}\Mod^b_{cell}$:
 	
 	\begin{itemize}
 		\item $\mbp^{iso}/\rho{\text -}\Mod^{b, \geq0}_{cell}\coloneqq \{E \in \mbp^{iso}/\rho{\text -}\Mod^b_{cell} : \pi_{p,q}(E)\cong 0 \:\mathrm{for} \: p<2q\},$
 		\item $\mbp^{iso}/\rho{\text -}\Mod^{b, \leq0}_{cell}\coloneqq \{E \in \mbp^{iso}/\rho{\text -}\Mod^b_{cell} : \pi_{p,q}(E)\cong 0 \:\mathrm{for} \: p>2q\},$ 
 		\item $\mbp^{iso}/\rho{\text -}\Mod^{\heartsuit}_{cell}\coloneqq \{E \in \mbp^{iso}/\rho{\text -}\Mod^b_{cell} : \pi_{p,q}(E)\cong 0 \:\mathrm{for} \: p\neq2q\}.$ 
 	\end{itemize}
 	\end{dfn}

\begin{prop}\label{heartmbpiso}
The functor 
$$\pi_{**}:\mbp^{iso}/\rho{\text -}\Mod^{\heartsuit}_{cell} \rightarrow \F{\text -}\Mod_*$$
is an equivalence of $\infty$-categories.
\end{prop}
\begin{proof}
	By Proposition \ref{hommbp}, for all $X$ and $Y$ in $\mbp^{iso}/\rho{\text -}\Mod^{\heartsuit}_{cell}$ and for all $n\geq0$, we have the following isomorphisms:
	\begin{align*}
		[\Sigma^{n,0}X,Y]_{\mbp^{iso}/\rho}\cong\Hom ^{n,0}_{\F}(\pi_{**}(X),\pi_{**}(Y))\cong
		\begin{cases}
			\Hom ^{0,0}_{\F}(\pi_{**}(X),\pi_{**}(Y)) & \mathrm{for} \: n=0\\
			0 & \mathrm{for} \:  n>0
		\end{cases}.
		\end{align*}
	
	This shows that the functor $\pi_{**}$ is fully faithful. Furthermore, since every object of $\F{\text -}\Mod_*$ is of the form $\bigoplus_{\alpha \in A}\Sigma^{2q_{\alpha},q_{\alpha}}\F \cong \pi_{**}(\bigvee_{\alpha \in A}\Sigma^{2q_{\alpha},q_{\alpha}}\mbp^{iso}/\rho)$, it follows that every such object belongs to the essential image of $\pi_{**}$. Consequently, we deduce that $\pi_{**}$ is essentially surjective, and therefore an equivalence.
	\end{proof}

\begin{prop}\label{mbpts}
The pair $(\mbp^{iso}/\rho{\text -}\Mod^{b, \geq0}_{cell},\mbp^{iso}/\rho{\text -}\Mod^{b, \leq0}_{cell})$ defines a bounded $t$-structure on the stable $\infty$-category $\mbp^{iso}/\rho{\text -}\Mod^{b}_{cell}$.
\end{prop}
\begin{proof}
	We will verify the properties outlined in \cite[Proposition 3.6]{GWX}. First, observe that $\mbp^{iso}/\rho{\text -}\Mod^{b, \geq0}_{cell}$ is closed under suspensions, $\mbp^{iso}/\rho{\text -}\Mod^{b, \leq0}_{cell}$ is closed under desuspensions, and both categories are closed under extensions. Moreover, we know that
	$$\mbp^{iso}/\rho{\text -}\Mod^{b}_{cell}=\bigcup_{n\in \Z}\mbp^{iso}/\rho{\text -}\Mod^{b, \geq n}_{cell},$$
	where 
	$$\mbp^{iso}/\rho{\text -}\Mod^{b, \geq n}_{cell}\coloneqq \{E \in \mbp^{iso}/\rho{\text -}\Mod^b_{cell} : \pi_{p,q}(E)\cong 0 \: \mathrm{for} \: p<2q+n\}.$$ 
	
	Now, consider $X$ in $\mbp^{iso}/\rho{\text -}\Mod^{b, \geq0}_{cell}$ and $Y$ in $\mbp^{iso}/\rho{\text -}\Mod^{b, \leq-1}_{cell}$. By Proposition \ref{hommbp}, we obtain:
	$$[X,Y]_{\mbp^{iso}/\rho}\cong \Hom ^{0,0}_{\F}(\pi_{**}(X),\pi_{**}(Y)) \cong 0.$$
	
	Next, let $X$ be an object in $\mbp^{iso}/\rho{\text -}\Mod^{b, \geq0}_{cell}$, and consider the epimorphism $\pi_{**}(X)\rightarrow \pi_{**}(X)_0$ that annihilates all elements of $\pi_{**}(X)$ in positive Chow-Novikov degrees. By Proposition \ref{heartmbpiso}, there exists an object $X_0$ in $\mbp^{iso}/\rho{\text -}\Mod^{\heartsuit}_{cell}$ such that $\pi_{**}(X_0) \cong\pi_{**}(X)_0$. By Proposition \ref{hommbp}, the epimorphism $\pi_{**}(X)\rightarrow \pi_{**}(X)_0$ is induced by a unique map $X \rightarrow X_0$, whose fiber belongs to $\mbp^{iso}/\rho{\text -}\Mod^{b, \geq1}_{cell}$. This completes the argument.
	\end{proof}

\begin{thm}\label{eqmbpisorho}
	There is an equivalence of stable $\infty$-categories:
	$$\mbp^{iso}/\rho{\text -}\Mod_{cell} \simeq \mathcal{D}(\F{\text -}\Mod_*).$$
\end{thm}
\begin{proof}
	By Propositions \ref{heartmbpiso} and \ref{mbpts}, the stable $\infty$-category $\mbp^{iso}/\rho{\text -}\Mod^b_{cell}$ is equipped with a bounded $t$-structure, whose heart is the abelian category $\F {\text -} \Mod_*$, which has enough projective objects. Moreover, for any $X$ and $Y$ in $\mbp^{iso}/\rho{\text -}\Mod^{\heartsuit}_{cell}$, we have the following isomorphism:
	$$[\Sigma^{-i}X,Y]_{\mbp^{iso}/\rho} \cong \Hom ^{-i,0}_{\F}(\pi_{**}(X),\pi_{**}(Y))\cong 0$$
	for all $i>0$, by Proposition \ref{hommbp}. Hence, \cite[Proposition 2.11]{GWX} implies that there is a $t$-exact equivalence of stable $\infty$-categories:
	$${\mathcal D}^b(\F{\text -}\Mod_*) \xrightarrow{\simeq}\mbp^{iso}/\rho{\text -}\Mod^b_{cell}.$$
	
	Next, by passing to the respective ind-completions, we obtain the following equivalence:
	$$\mbp^{iso}/\rho{\text -}\Mod_{cell} \simeq \ind(\mbp^{iso}/\rho{\text -}\Mod^b_{cell}) \simeq \ind({\mathcal D}^b(\F{\text -}\Mod_*))=\mathcal{D}(\F{\text -}\Mod_*),$$
	which concludes the proof.
	\end{proof}

\section{The special fiber}

In the previous section, we considered the parameter $\rho$ in the category of real isotropic motivic spectra, and we identified $\mbp^{iso}/\rho$-cellular modules with bigraded $\F$-vector spaces. We are now ready to study the special fiber of this deformation in $\SH^{iso}_{cell}(\R)$, namely the category of $\un^{iso}/\rho$-cellular modules.

\begin{lem}\label{injmbp}
	Let $I$ be an object in $\mbp^{iso}/\rho {\text -} \Mod_{cell}$. Then, there is an isomorphism of left $\G_{**}$-comodules:
	$$\mbp_{**}^{iso}(I) \cong \G_{**}\otimes_{\F}\pi_{**}(I).$$
\end{lem}
\begin{proof}
	It follows from Thorem \ref{eqmbpisorho} that $I \simeq \bigvee_{\alpha \in A}\Sigma^{p_{\alpha},q_{\alpha}}\mbp^{iso}/\rho$. Thus, by Proposition \ref{pimbpisorho} and \cite[Lemma 5.4]{HKO}, we obtain the following chain of isomorphisms:
	\begin{align*}
	\mbp_{**}^{iso}(I) &= \pi_{**}(\mbp^{iso}\wedge I)\\
	&\cong \pi_{**}((\mbp^{iso}\wedge \mbp^{iso}/\rho) \wedge_{\mbp^{iso}/\rho} I)\\
	&\cong \pi_{**}(\mbp^{iso}\wedge \mbp^{iso}/\rho) \otimes_{\pi_{**}(\mbp^{iso}/\rho)}\pi_{**}(I)\\
	&\cong \G_{**}\otimes_{\F}\pi_{**}(I),
	\end{align*}
which is what we wanted to prove.
	\end{proof}

\begin{lem}\label{unisoinj}
Let $X$ be an object in $\un^{iso}/\rho{\text -}\Mod_{cell}$ and $I$ an object in $\mbp^{iso}/\rho {\text -} \Mod_{cell}$. Then, there is an isomorphism:
$$[X,I]_{\un^{iso}/\rho} \cong \Hom^{0,0}_{\G_{**}}(\mbp^{iso}_{**}(X),\mbp^{iso}_{**}(I)).$$
\end{lem}
\begin{proof}
	Since $X$ is $\un^{iso}/\rho$-cellular, it follows that $X\wedge_{\un^{iso}/\rho}\mbp^{iso}/\rho$ is $\mbp^{iso}/\rho$-cellular. Therefore, by Corollary \ref{algmbp}, Theorem \ref{eqmbpisorho} and Lemma \ref{injmbp}, we have the following sequence of isomorphisms:
	\begin{align*}
	[X,I]_{\un^{iso}/\rho} &\cong [X\wedge_{\un^{iso}/\rho}\mbp^{iso}/\rho,I]_{\mbp^{iso}/\rho}\\
	& \cong \Hom^{0,0}_{\F}(\pi_{**}(X\wedge_{\un^{iso}/\rho}\mbp^{iso}/\rho),\pi_{**}(I))\\
	& \cong \Hom^{0,0}_{\F}(\pi_{**}(X\wedge\mbp^{iso}),\pi_{**}(I))\\
	&\cong \Hom^{0,0}_{\G_{**}}(\mbp^{iso}_{**}(X),\G_{**}\otimes_{\F}\pi_{**}(I))\\
	&\cong \Hom^{0,0}_{\G_{**}}(\mbp^{iso}_{**}(X),\mbp^{iso}_{**}(I)).
	\end{align*}

This concludes the argument.
	\end{proof}

\begin{dfn}
	\normalfont
	The Chow-Novikov degree of $\mbp^{iso}_{p,q}(-)$ is defined as the integer $p -2q$. 
	
	Denote by $\un^{iso}/\rho{\text -}\Mod^b_{cell}$ the full subcategory of bounded $\un^{iso}/\rho$-cellular modules, that is, objects in $\un^{iso}/\rho{\text -}\Mod_{cell}$ whose $\mbp^{iso}$-homology groups are non-trivial only for a finite number of Chow-Novikov degrees. 
\end{dfn}

\begin{thm}\label{speseq}
Let $X$ and $Y$ be objects in $\un^{iso}/\rho{\text -}\Mod^b_{cell}$. Then, there is a strongly convergent spectral sequence:
$$E_2^{s,t,u}\cong \ext^{s,t,u}_{\G_{**}}(\mbp^{iso}_{**}(X),\mbp^{iso}_{**}(Y)) \Longrightarrow [\Sigma^{t-s,u}X,Y]_{\un^{iso}/\rho}.$$
\end{thm}
\begin{proof}
	Since $\mbp$ is a cellular spectrum, for any $\un^{iso}/\rho$-cellular module $Y$, we can consider the following Postnikov system in $\un^{iso}/\rho{\text -}\Mod_{cell}$:
	$$
	\xymatrix{
		\dots \ar@{->}[r] &  \overline{\mbp^{iso}}^{\wedge s} \wedge Y \ar@{->}[r] \ar@{->}[d] &  \dots \ar@{->}[r]   & \overline{\mbp^{iso}} \wedge Y \ar@{->}[r] \ar@{->}[d]	 & Y \ar@{->}[d]  \\
		&	\mbp^{iso}\wedge \overline{ \mbp^{iso}}^{\wedge s} \wedge Y \ar@{->}[ul]^{\Sigma} & & \mbp^{iso} \wedge \overline{\mbp^{iso}} \wedge Y \ar@{->}[ul]^{\Sigma}  &	\mbp^{iso} \wedge Y \ar@{->}[ul]^{\Sigma} 
	}.
	$$

	For any $X$ in $\un^{iso}/\rho{\text -}\Mod_{cell}$, applying the functor $[\Sigma^{**}X,-]_{\un^{iso}/\rho}$ to the previous Postnikov system yields an isotropic Adams-Novikov spectral sequence whose $E_1$-page is given by:
	$$E_1^{s,t,u} \cong [\Sigma^{t-s,u}X,\mbp^{iso}\wedge (\overline{\mbp^{iso}})^{\wedge s} \wedge Y]_{\un^{iso}/\rho},$$
	with the first differential:
	$$d_1^{s,t,u}:E_1^{s,t,u} \rightarrow E_1^{s+1,t,u}.$$
	
	Since $Y$ is also $\un^{iso}/\rho$-cellular, it follows that $\mbp^{iso}\wedge (\overline{\mbp^{iso}})^{\wedge s} \wedge Y$ is $\mbp^{iso}/\rho$-cellular for all $s \geq 0$. Therefore, by Lemmas \ref{unisoinj} and \ref{overmbp}, there exists an isomorphism:
	$$E_1^{s,t,u} \cong \Hom ^{t,u}_{\G_{**}}(\mbp^{iso}_{**}(X),\G_{**} \otimes_{\F} \overline{\G_{**}}^{\otimes s} \otimes_{\F} \mbp^{iso}_{**}(Y)),$$
	from which we deduce that the $E_2$-page admits the following standard description:
	$$E_2^{s,t,u} \cong \ext^{s,t,u}_{\G_{**}}(\mbp^{iso}_{**}(X),\mbp^{iso}_{**}(Y)).$$
	
	As constructed, the isotropic Adams-Novikov spectral sequence conditionally converges to the groups $[\Sigma^{t-s,u}X,Y^{\wedge}_{\mbp}]_{\un^{iso}/\rho}$. Since $\un^{iso}/\rho$ is $\mbp$-complete by Corollary \ref{mbpcomp}, and $Y$ is a $\un^{iso}/\rho$-module, it follows that $Y$ is also $\mbp$-complete, and so $Y^{\wedge}_{\mbp} \simeq Y$.
	
	To establish the strong convergence of the spectral sequence, we have to use the boundedness assumption on $X$ and $Y$. Specifically, by hypothesis, we know that $\mbp^{iso}_{**}(X)$ and $\mbp^{iso}_{**}(Y)$ are concentrated in Chow-Novikov degrees $[a,b]$ and $[c,d]$ respectively, for some integers $a$, $b$, $c$ and $d$. This forces the $E_1$-page to be trivial outside the range $c-b+2u \leq t \leq d-a+2u$. Since the $r$-th differential is given by:
	$$d_r^{s,t,u}:E_r^{s,t,u} \rightarrow E_r^{s+r,t+r-1,u},$$
	we deduce that it must be trivial for $r > d-a-c+b+1$. Consequently, the spectral sequence collapses at the $E_{d-a-c+b+2}$-page, implying strong convergence, as desired.
	\end{proof}

\begin{dfn}
	\normalfont
We define the following full subcategories of $\un^{iso}/\rho{\text -}\Mod^b_{cell}$:
	\begin{itemize}
		\item $\un^{iso}/\rho{\text -}\Mod^{b, \geq0}_{cell}\coloneqq \{E \in \un^{iso}/\rho{\text -}\Mod^b_{cell} : \mbp^{iso}_{p,q}(E)\cong 0 \:\mathrm{for} \: p<2q\},$ 
		\item $\un^{iso}/\rho{\text -}\Mod^{b, \leq0}_{cell}\coloneqq \{E \in \un^{iso}/\rho{\text -}\Mod^b_{cell} : \mbp^{iso}_{p,q}(E)\cong 0 \:\mathrm{for} \: p>2q\},$
		\item $\un^{iso}/\rho{\text -}\Mod^{\heartsuit}_{cell}\coloneqq \{E \in \un^{iso}/\rho{\text -}\Mod^b_{cell} : \mbp^{iso}_{p,q}(E)\cong 0 \:\mathrm{for} \: p\neq2q\}.$
	\end{itemize}
\end{dfn}

\begin{cor}\label{hom}
	Let $X$ be an object in $\un^{iso}/\rho{\text -}\Mod^{b, \geq0}_{cell}$ and $Y$ an object in $\un^{iso}/\rho{\text -}\Mod^{b, \leq0}_{cell}$. Then, there is an isomorphism:
	$$[X,Y]_{\un^{iso}/\rho} \cong \Hom ^{0,0}_{\G_{**}}(\mbp^{iso}_{**}(X),\mbp^{iso}_{**}(Y)).$$
\end{cor}
\begin{proof}
	We use the isotropic Adams-Novikov spectral sequence developed in Theorem \ref{speseq}, whose $E_1$-page is given by:
	$$E_1^{s,t,u} \cong \Hom ^{t,u}_{\G_{**}}(\mbp^{iso}_{**}(X),\G_{**} \otimes_{\F} \overline{\G_{**}}^{\otimes s} \otimes_{\F} \mbp^{iso}_{**}(Y)).$$
	
	Since $X$ is concentrated in non-negative Chow-Novikov degree and $Y$ in non-positive Chow-Novikov degree, it follows that $E_1^{s,t,u}\cong 0$ for all $s \geq 0$ and $t > 2u$. 
	
	Note that the only groups contributing to $[X,Y]_{\un^{iso}/\rho}$ are of the form $E_1^{t,t,0}$ for all $t \geq 0$, and these are trivial for $t>0$. Furthermore, for all $r\geq 1$, we have the following complex:
	$$E_r^{-r,-r+1,0}\xrightarrow{d_r}E_r^{0,0,0}\xrightarrow{d_r}E_r^{r,r-1,0},$$
	whose homology in the middle is $E_{r+1}^{0,0,0}$. Since the two ends of the complex are trivial for $r\geq 2$, we obtain an isomorphism $E_{r+1}^{0,0,0}\cong E_{r}^{0,0,0}$ for all $r\geq 2$. Therefore, we have:
	$$[X,Y]_{\un^{iso}/\rho} \cong E_{\infty}^{0,0,0}\cong E_2^{0,0,0} \cong \ext^{0,0,0}_{\G_{**}}(\mbp^{iso}_{**}(X),\mbp^{iso}_{**}(Y)) \cong \Hom^{0,0}_{\G_{**}}(\mbp^{iso}_{**}(X),\mbp^{iso}_{**}(Y)),$$
	which completes the proof.
	\end{proof}

\begin{cor}\label{morcor}
	Let $X$ and $Y$ be objects in $\un^{iso}/\rho{\text -}\Mod^{\heartsuit}_{cell}$. Then, there is an isomorphism:
	$$[\Sigma^{t,u}X,Y]_{\un^{iso}/\rho} \cong \ext^{2u-t,2u,u}_{\G_{**}}(\mbp^{iso}_{**}(X),\mbp^{iso}_{**}(Y)).$$
\end{cor}
\begin{proof}
	Since both $X$ and $Y$ are concentrated in Chow-Novikov degree 0, the only non-trivial groups on the $E_1$-page of the isotropic Adams-Novikov spectral sequence are of the form $E^{s,2u,u}_1$ for all $s \geq 0$ and $u \in \Z$. Therefore, the complex 
	$$E_r^{s-r,2u-r+1,u}\xrightarrow{d_r}E_r^{s,2u,u}\xrightarrow{d_r}E_r^{s+r,2u+r-1,u}$$
	has trivial ends for all $r \geq 2$, so we get an isomorphism $E_{r+1}^{s,2u,u}\cong E_r^{s,2u,u}$. Hence, the spectral sequence collapses at the second page, and we conclude that
	$$[\Sigma^{t,u}X,Y]_{\un^{iso}/\rho} \cong E_{\infty}^{2u-t,2u,u} \cong E_2^{2u-t,2u,u} \cong \ext^{2u-t,2u,u}_{\G_{**}}(\mbp^{iso}_{**}(X),\mbp^{iso}_{**}(Y)),$$
	which is what we aimed to show.
	\end{proof}

\begin{lem}\label{ext}
	Let $M$ be a left $\G_{**}$-comodule that is a finite $\F$-module concentrated in Chow-Novikov degree 0. Then, there exists an object $X$ in $\un^{iso}/\rho{\text -}\Mod^{\heartsuit}_{cell}$ such that $M \cong \mbp^{iso}_{**}(X)$.
\end{lem}
\begin{proof}
	Since $M$ is finite, \cite[Theorem 3.3]{L} implies that there exists a finite filtration of subcomodules
	$$0 \cong M_0 \subset M_1 \subset \dots \subset M_n \cong M$$
	with graded components $M_i/M_{i-1} \cong \Sigma^{2q_i,q_i} \F$ for all $1\leq i\leq n$.
	 
	We proceed by induction on $i$. For $i=1$, we have:
	$$M_1\cong M_1/M_0 \cong \Sigma^{2q_1,q_1} \F \cong \mbp^{iso}_{**}(\Sigma^{2q_1,q_1}\un^{iso}/\rho).$$
	
    Now, suppose we have an object $X_{i-1}$ in $\un^{iso}/\rho{\text-}\Mod^{\heartsuit}_{cell}$ such that $M_{i-1} \cong \mbp^{iso}_{**}(X_{i-1})$. By Corollary \ref{morcor}, the short exact sequence
	$$0 \rightarrow M_{i-1} \rightarrow M_i
	\rightarrow \Sigma^{2q_i,q_i} \F \rightarrow 0$$
	corresponds to a map $f_i$ in $$[\Sigma^{2q_i-1,q_i}\un^{iso}/\rho,X_{i-1}]_{\un^{iso}/\rho} \cong \ext_{\G_{**}}^{1,0,0}(\Sigma^{2q_i,q_i} \F,M_{i-1}),$$
    so we can define an object $X_i$ in $\un^{iso}/\rho{\text -}\Mod^{\heartsuit}_{cell}$ by:
	$$X_i \coloneqq \cof(\Sigma^{2q_i-1,q_i}\un^{iso}/\rho \xrightarrow{f_i} X_{i-1}).$$ 
	
	The associated cofiber sequence induces a short exact sequence of $\mbp^{iso}$-homology groups:
	$$0 \rightarrow  M_{i-1} \rightarrow \mbp^{iso}_{**}(X_i) \rightarrow \Sigma^{2q_i,q_i} \F  \rightarrow 0.$$
	
	By Remark \ref{morcor} and the construction of $X_i$, it follows that the latter extension coincides with the one corresponding to $f_i$. Therefore, we obtain an isomorphism $\mbp^{iso}_{**}(X_i) \cong M_i$, which completes the proof.
	\end{proof}

\begin{lem}\label{filtcol}
	Let $\{X_{\alpha}\}$ be a filtered system in $\un^{iso}/\rho{\text -}\Mod^{\heartsuit}_{cell}$. Then, the colimit $\colim X_{\alpha}$, taken in $\un^{iso}/\rho{\text -}\Mod_{cell}$, belongs to $\un^{iso}/\rho{\text -}\Mod^{\heartsuit}_{cell}$.	
\end{lem}
\begin{proof}
	Since $\mbp^{iso}$-homology preserves filtered colimits, we immediately see that $\colim X_{\alpha}$, where the colimit is taken in $\un^{iso}/\rho{\text -}\Mod_{cell}$, has $\mbp^{iso}$-homology concentrated in Chow-Novikov degree 0. This concludes the proof.
	\end{proof}

\begin{prop}\label{heart}
	The functor 
	$$\mbp^{iso}_{**}:\un^{iso}/\rho{\text -}\Mod^{\heartsuit}_{cell} \rightarrow \A_{*}{\text -}\Com_{*}$$
	is an equivalence of $\infty$-categories.
\end{prop}
\begin{proof}
		By Corollary \ref{hom}, for all $X$ and $Y$ in $\un^{iso}/\rho{\text -}\Mod^{\heartsuit}_{cell}$ and for all $n\geq0$, we have the following isomorphisms:
	\begin{align*}
		[\Sigma^{n,0}X,Y]_{\un^{iso}/\rho}\cong\Hom ^{n,0}_{\G_{**}}(\mbp^{iso}_{**}(X),\mbp^{iso}_{**}(Y))\cong
		\begin{cases}
			\Hom ^{0,0}_{\G_{**}}(\mbp^{iso}_{**}(X),\mbp^{iso}_{**}(Y)) & \mathrm{for} \: n=0\\
			0 & \mathrm{for} \:  n>0
		\end{cases}
	\end{align*}
	from which it follows that the functor $\mbp^{iso}_{**}$ is fully faithful. On the other hand, every object $M$ in $\A_{*}{\text -}\Com_{*}$ is a filtered colimit of graded comodules $M_{\alpha}$, where each $M_{\alpha}$ is finite as an $\F$-module. By Lemma \ref{ext}, each of these $M_{\alpha}$ lies in the essential image of $\mbp^{iso}_{**}$ and, by Lemma \ref{filtcol}, the subcategory $\un^{iso}/\rho{\text -}\Mod^{\heartsuit}_{cell}$ is closed under filtered colimits in $\un^{iso}/\rho{\text -}\Mod_{cell}$. We deduce that $M\cong \colim M_{\alpha}$ is also in the essential image of $\mbp^{iso}_{**}$. Hence, $\mbp^{iso}_{**}$ is essentially surjective, and therefore an equivalence.
	\end{proof}

\begin{prop}\label{tri}
The pair $(\un^{iso}/\rho{\text -}\Mod^{b, \geq0}_{cell},\un^{iso}/\rho{\text -}\Mod^{b, \leq0}_{cell})$ defines a bounded $t$-structure on the stable $\infty$-category $\un^{iso}/\rho{\text -}\Mod^{b}_{cell}$.
\end{prop}
\begin{proof}
	By definition, $\un^{iso}/\rho{\text -}\Mod^{b, \geq0}_{cell}$ is closed under suspensions, $\un^{iso}/\rho{\text -}\Mod^{b, \leq0}_{cell}$ is closed under desuspensions, and both are closed under extensions. Moreover, we have the identification:
	$$\un^{iso}/\rho{\text -}\Mod^{b}_{cell}=\bigcup_{n\in \Z}\un^{iso}/\rho{\text -}\Mod^{b, \geq n}_{cell},$$
	where 
	$$\un^{iso}/\rho{\text -}\Mod^{b, \geq n}_{cell}\coloneqq \{E \in \un^{iso}/\rho{\text -}\Mod^{b}_{cell} : \mbp^{iso}_{p,q}(E)\cong 0 \: \mathrm{for} \: p<2q+n\}.$$ 
	
	Now, let $X$ be in $\un^{iso}/\rho{\text -}\Mod^{b, \geq0}_{cell}$ and $Y$ in $\un^{iso}/\rho{\text -}\Mod^{b, \leq-1}_{cell}$. By Corollary \ref{hom}, we obtain:
	$$[X,Y]_{\un^{iso}/\rho}\cong \Hom ^{0,0}_{\G_{**}}(\mbp^{iso}_{**}(X),\mbp^{iso}_{**}(Y)) \cong 0.$$
	
	To finish, take $X$ in $\un^{iso}/\rho{\text -}\Mod^{b, \geq0}_{cell}$ and consider the epimorphism $\mbp^{iso}_{**}(X)\rightarrow \mbp^{iso}_{**}(X)_0$ that annihilates all elements of $\mbp^{iso}_{**}(X)$ in positive Chow-Novikov degrees. By Proposition \ref{heart}, there exists an object $X_0$ in $\un^{iso}/\rho{\text -}\Mod^{\heartsuit}_{cell}$ such that $\mbp^{iso}_{**}(X_0) \cong\mbp^{iso}_{**}(X)_0$. 
	
	By Corollary \ref{hom}, the epimorphism $\mbp^{iso}_{**}(X)\rightarrow \mbp^{iso}_{**}(X)_0$ is induced by a unique map $X \rightarrow X_0$, whose fiber belongs to $\un^{iso}/\rho{\text -}\Mod^{b, \geq1}_{cell}$. Thus, \cite[Proposition 3.6]{GWX} concludes the proof.
	\end{proof}

\begin{thm}\label{specfib}
	There is an equivalence of stable $\infty$-categories:
	$$\un^{iso}/\rho{\text -}\Mod_{cell} \simeq {\mathcal D}(\A_*{\text -}\Com_*).$$
\end{thm}
\begin{proof}
		By Propositions \ref{heart} and \ref{tri}, the stable $\infty$-category $\un^{iso}/\rho{\text -}\Mod^b_{cell}$ is equipped with a bounded $t$-structure, whose heart is the abelian category $\A_* {\text -} \Com_*$, which has enough injective objects. 
		
		Now, let $X$ and $Y$ be objects in $\un^{iso}/\rho{\text -}\Mod^{\heartsuit}_{cell} $ such that $\mbp^{iso}_{**}(Y)$ is an injective $\G_{**}$-comodule. By Corollary \ref{morcor}, we have the following isomorphisms:
		$$[\Sigma^{-i}X,Y]_{\un^{iso}/\rho} \cong \ext^{i,0,0}_{\G_{**}}(\mbp^{iso}_{**}(X),\mbp^{iso}_{**}(Y)) \cong 0$$
		for all $i>0$, since $\mbp^{iso}_{**}(Y)$ is injective. Hence, \cite[Proposition 2.11]{GWX} implies that there is a $t$-exact equivalence of stable $\infty$-categories:
	$${\mathcal D}^b(\A_*{\text -}\Com_*) \xrightarrow{\simeq}\un^{iso}/\rho{\text -}\Mod^b_{cell}.$$
	
	By passing to the respective ind-completions, we obtain the equivalence:
	$$\un^{iso}/\rho{\text -}\Mod_{cell} \simeq \ind(\un^{iso}/\rho{\text -}\Mod^b_{cell}) \simeq \ind({\mathcal D}^b(\A_*{\text -}\Com_*))={\mathcal D}(\A_*{\text -}\Com_*),$$
	which completes the proof.
\end{proof}

\section{The generic fiber}

Once we have understood the structure of the special fiber, we can proceed to study the generic fiber of the deformation in $\SH^{iso}(\R)$ induced by the parameter $\rho$. This is the main goal of this section.

Denote by $\rr:\SH(\R)\rightarrow \SH$ the real realization functor, which sends an $\R$-scheme $X$ to the topological space $X(\R)$ of its real points.

\begin{prop}\label{unisos}
	There is an equivalence of $E_{\infty}$-algebras in $\SH$:
	$$\rr(\un^{iso}) \simeq \sph.$$
\end{prop}
\begin{proof}
	Since real realization preserves colimits, we obtain the following equivalence in $\SH$:
	$$\rr(\un^{iso}) \simeq \colim_n\cof(\Sigma^{\infty}_+\ce(\rr(Q_{\rho^n}))\rightarrow \sph).$$
	
	Since the quadric $Q_{\rho^n}$ is anisotropic, it has no real points, implying that $\rr(Q_{\rho^n})=\emptyset$ for all $n \geq 0$. Therefore, we have an equivalence $\Sigma^{\infty}_+\ce(\rr(Q_{\rho^n})) \simeq 0$ for all $n$, from which we deduce that $\rr(\un^{iso}) \simeq \sph$.
	\end{proof}

\begin{cor}
The real realization functor $\rr:\SH(\R)\rightarrow \SH$ factors through $\SH^{iso}(\R)$.
\end{cor}
\begin{proof}
	It follows from Proposition \ref{unisos} that:
	$$\rr(E^{iso}) = \rr(\un^{iso} \wedge E) \simeq \rr(\un^{iso})\wedge \rr(E) \simeq \sph \wedge \rr(E) \simeq \rr(E)$$
	for all real motivic spectra $E$. Since $\SH^{iso}(\R)$ is the localization of $\SH(\R)$ obtained by inverting the maps $E \rightarrow E^{iso}$, we deduce that $\rr:\SH(\R)\rightarrow \SH$ factors as:
	$$\SH(\R)\xrightarrow{L^{iso}} \SH^{iso}(\R) \xrightarrow{\rr} \SH$$
	which is what we aimed to show.
	\end{proof}

\begin{thm}\label{equniso}
There is an equivalence of stable $\infty$-categories:
$$\un^{iso}[\rho^{-1}]{\text -}\Mod \simeq \SH.$$
\end{thm}
\begin{proof}
	By \cite[Theorem 35 and Proposition 36]{B}, we know that $\SH(\R)[\rho^{-1}] \simeq \SH$, and that the $\rho$-periodization functor $\SH(\R) \rightarrow \SH(\R)[\rho^{-1}]$ coincides with real realization. Hence, by Proposition \ref{unisos}, the fully faithful functor
	$$\un^{iso}[\rho^{-1}]{\text -}\Mod \hookrightarrow \SH(\R)[\rho^{-1}] \simeq \SH$$
	is an equivalence. This concludes the proof.
	\end{proof}

\begin{prop}\label{mbpisohft}
	There is an equivalence of $E_{\infty}$-algebras in $\SH$:
	$$\rr(\mbp^{iso}) \simeq \hzt.$$
\end{prop}
\begin{proof}
	The homotopy groups of $\rr(\mbp^{iso})$ can be computed as follows:
	$$\pi_*(\rr(\mbp^{iso}))\cong \pi_{*,0}(\mbp^{iso}[\rho^{-1}])\cong \colim(\pi_{*,0}(\mbp^{iso}) \xrightarrow{\cdot \rho}\pi_{*-1,-1}(\mbp^{iso})\xrightarrow{\cdot \rho}\dots).$$
	
	Therefore, by Theorem \ref{pimbp}, we obtain an isomorphism $\pi_*(\rr(\mbp^{iso}))\cong \F$. This implies that $\rr(\mbp^{iso})\simeq \hzt$.
	\end{proof}

\begin{prop}\label{rehg}
	For all $p \geq 2q$, the real realization functor $\rr$ induces an isomorphism of homotopy groups:
	$$\pi_{p,q}(\un^{iso}) \cong \pi_{p-q},$$
	where $\pi_*$ are the homotopy groups of the sphere spectrum $\sph$.
\end{prop}
\begin{proof}
	The cofiber sequence
	$$\Sigma^{-1,-1}\un^{iso} \xrightarrow{\rho}\un^{iso} \rightarrow \un/\rho$$
	induces a long exact sequence of homotopy groups
	$$\dots \rightarrow \pi_{p,q-1}(\un^{iso}/\rho) \rightarrow \pi_{p,q}(\un^{iso}) \xrightarrow{\cdot \rho}\pi_{p-1,q-1}(\un^{iso})\rightarrow \pi_{p-1,q-1}(\un^{iso}/\rho) \rightarrow \dots.$$
	
	By Theorem \ref{mustlab}, we know that for $p \geq 2q$ both $\pi_{p,q-1}(\un^{iso}/\rho)$ and $\pi_{p-1,q-1}(\un^{iso}/\rho)$ are zero, as they are in positive Chow-Novikov degrees. It follows that the multiplication by $\rho$ gives an isomorphism $\pi_{p,q}(\un^{iso}) \xrightarrow{\cong} \pi_{p-1,q-1}(\un^{iso})$ for $p \geq 2q$.
	
	Since $\pi_{p,q}(\un^{iso})$ is in non-negative Chow-Novikov degree, then also $\pi_{p-1,q-1}(\un^{iso})$ is so. Hence, by induction, for all $n >0$, we get an isomorphism $\pi_{p,q}(\un^{iso}) \xrightarrow{\cong} \pi_{p-n,q-n}(\un^{iso})$ that is the multiplication by $\rho^n$. 
	
	We conclude by noticing that
	$$\pi_{p-q} \cong \pi_{p-q}(\rr(\un^{iso}))\cong \pi_{p-q,0}(\un^{iso}[\rho^{-1}])\cong \colim(\pi_{p,q}(\un^{iso}) \xrightarrow{\cdot \rho}\pi_{p-1,q-1}(\un^{iso})\xrightarrow{\cdot \rho}\dots) \cong \pi_{p,q}(\un^{iso})$$
	for $p \geq 2q$, by Proposition \ref{unisos}.
	\end{proof}

\begin{rem}
\normalfont
We can summarize what we have proved so far by saying that the category $\SH^{iso}_{cell}(\R) \coloneqq \un^{iso}{\text -}\Mod_{cell}$ is a one-parameter deformation of $\SH$ with parameter $\rho$ and special fiber given by derived comodules over the dual Steenrod algebra:
$$\SH \xleftarrow{\rho^{-1}} \SH^{iso}_{cell}(\R) \xrightarrow{\rho=0} {\mathcal D}(\A_*{\text -}\Com_*).$$

This description of the category of real isotropic cellular spectra is reminiscent of the one of $\F$-synthetic spectra introduced in \cite{P}. In the next section, we show that these categories are in fact equivalent, thereby bridging real isotropic motivic homotopy theory and $\F$-synthetic homotopy theory.
\end{rem}

\section{Comparison with $\F$-synthetic spectra}

The main purpose of this section is to compare real isotropic cellular spectra with $\F$-synthetic spectra. Before that, let us recall from \cite{P} the definition of $\F$-synthetic spectra.

\begin{dfn}
	\normalfont
A spectrum $X$ is called finite $\hzt$-projective if it is finite and $H_*(X;\F)$ is a finite $\F$-vector space. We denote by $\Sp^{fp}_{\hzt}$ the full subcategory of spectra spanned by finite $\hzt$-projective spectra. Then, the category of $\F$-synthetic spectra, denoted by $\syn_{\F}$, is defined as the category of spherical sheaves of spectra on $\Sp^{fp}_{\hzt}$ with respect to the $(\hzt)_*$-surjection topology. In brief, we have:
$$\syn_{\F} \coloneqq \sh^{\Sp}_{\Sigma}(\Sp^{fp}_{\hzt}).$$
\end{dfn}
 
In order to compare real isotropic spectra with $\F$-synthetic spectra, we need to introduce an intermediate category.

\begin{dfn}
	\normalfont
An object $X$ in $\SH^{iso}_{cell}(\R)$ is called finite $\mbp^{iso}$-projective if it is finite and $\mbp^{iso}_{**}(X)$ is a free $\mbp^{iso}_{**}$-module, finitely generated by classes in Chow-Novikov degree 0. We denote by $\Sp^{fp}_{\mbp^{iso}}$ the full subcategory of $\SH^{iso}_{cell}(\R)$ spanned by finite $\mbp^{iso}$-projective isotropic cellular spectra, and by $\sh^{\Sp}_{\Sigma}(\Sp^{fp}_{\mbp^{iso}})$ the category of spherical sheaves of spectra on $\Sp^{fp}_{\mbp^{iso}}$ with respect to the $\mbp^{iso}_{**}$-surjection topology.
\end{dfn}

\begin{prop}\label{cruc}
	If $M$ is a finite $\mbp^{iso}$-projective motivic spectrum, then $\rr(M)$ is a finite $\hzt$-projective spectrum. Moreover, the morphism induced by real realization
	$$\mbp^{iso}_{2q,q}(M) \rightarrow H_q(\rr(M);\F)$$ 
	is an isomorphism.
\end{prop}
\begin{proof}
	By Lemma \ref{modcell} we have an equivalence:
	$$\mbp^{iso} \wedge M \simeq \bigvee_{\alpha \in A}\Sigma^{2q_{\alpha},q_{\alpha}}\mbp^{iso}$$
	for some finite set $A$. Then, it follows from Proposition \ref{mbpisohft} that: 
	$$\hzt \wedge \rr(M) \simeq \rr(\mbp^{iso} \wedge M) \simeq \bigvee_{\alpha \in A}\Sigma^{q_{\alpha}}\hzt,$$
	which proves the first part of the statement.
	
	To show the second part, notice that by Theorem \ref{pimbp} we have an isomorphism of $\mbp^{iso}_{**}$-modules:
	$$\mbp^{iso}_{**}(M) \cong \bigoplus_{\alpha \in A}\mbp^{iso}_{**}\cdot \{x_{\alpha}\} \cong \bigoplus_{\alpha \in A}\F[\rho]\cdot \{x_{\alpha}\}$$
	with $x_{\alpha} \in \mbp^{iso}_{2q_{\alpha},q_{\alpha}}(M)$. By restricting to Chow-Novikov degree 0, we obtain:
	$$\mbp^{iso}_{2*,*}(M) \cong  \bigoplus_{\alpha \in A}\F\cdot \{x_{\alpha}\} \cong H_*(\rr(M);\F)$$
    which completes the proof.
	\end{proof}

\begin{rem}
\normalfont
In \cite[Section 7.2]{P}, the category of finite $\mgl$-projective motivic spectra over $\C$, denoted by $\Sp^{fp}_{\mgl}$, is related to the category of finite even $\MU$-projective spectra, denoted by $\Sp^{fpe}_{\MU}$, via the complex realization functor $\rc:\SH(\C) \rightarrow \SH$, which maps $\mgl$ to $\MU$. More precisely, we have a complex realization functor $\rc:\Sp^{fp}_{\mgl}\rightarrow \Sp^{fpe}_{\MU}$ that induces, in turn, a functor between spherical sheaves:
$$\rc^*:\sh^{\Sp}_{\Sigma}(\Sp^{fp}_{\mgl}) \rightarrow \sh^{\Sp}_{\Sigma}(\Sp^{fpe}_{\MU}),$$
which can be described as the unique cocontinuous functor mapping $\Sigma^{\infty}_+y(M)$ to $\Sigma^{\infty}_+y(\rc(M))$, where $y(-)$ denotes the representable sheaf of spaces and $M$ is an object in $\Sp^{fp}_{\mgl}$. The existence of this functor crucially relies on \cite[Lemma 7.9]{P}.

 The situation for real isotropic motivic spectra is formally the same once we replace $\SH(\C)$ with $\SH^{iso}(\R)$, $\rc$ with $\rr$, $\mgl$ with $\mbp^{iso}$ and $\MU$ with $\hzt$, by Proposition \ref{mbpisohft} and Proposition \ref{cruc}, which is the analogue of \cite[Lemma 7.9]{P}. In the end, we obtain a functor:
 $$\rr^*:\sh^{\Sp}_{\Sigma}(\Sp^{fp}_{\mbp^{iso}}) \rightarrow \sh^{\Sp}_{\Sigma}(\Sp^{fp}_{\hzt})$$
that is the unique cocontinuous functor mapping $\Sigma^{\infty}_+y(M)$ to $\Sigma^{\infty}_+y(\rr(M))$, where $M$ is an object in $\Sp^{fp}_{\mbp^{iso}}$.
 \end{rem}

\begin{dfn}
\normalfont
For any object $X$ in $\SH^{iso}_{cell}(\R)$, denote by $\Upsilon X$ the presheaf of spectra on $\Sp^{fp}_{\mbp^{iso}}$ defined by $\Upsilon X(M) \coloneqq \mathrm{map}(M,X)$, where $ \mathrm{map}$ denotes the mapping spectrum in $\SH^{iso}_{cell}(\R)$.
\end{dfn}

\begin{lem}\label{upsidupsi}
For any object $X$ in $\SH^{iso}_{cell}(\R)$, the presheaf of spectra $\Upsilon X$ belongs to $\sh^{\Sp}_{\Sigma}(\Sp^{fp}_{\mbp^{iso}})$. Moreover, if $\mbp^{iso}_{**}(X)$ is concentrated in non-negative Chow-Novikov degrees, then $\Upsilon X \simeq \Sigma^{\infty}_+y(X)$.
\end{lem}
\begin{proof}
	The proof is analogous to \cite[Lemmas 7.17, 7.18 and 7.19]{P}. 
	\end{proof}

By Lemma \ref{upsidupsi}, we obtain a cocontinuous functor:
$$\Theta^{*}\coloneqq \rr^* \circ \Upsilon: \SH^{iso}_{cell}(\R) \rightarrow \syn_{\F}$$
from real isotropic cellular spectra to $\F$-synthetic spectra. Our goal is to prove that this functor is indeed an equivalence. 

\begin{lem}\label{spheres}
For all $p$, $q \in \Z$, there is an equivalence of $\F$-synthetic spectra:
$$\Theta^*(\Sigma^{p,q}\un^{iso}) \simeq \sph^{p-q,q},$$
where $\sph^{t,w}$ denotes the bigraded $\F$-synthetic sphere $\Sigma^{t-w}\Sigma^{\infty}_+y(\sph^w)$.
\end{lem}
\begin{proof}
	Since $\Theta^*$ is an exact functor, it is enough to prove the equivalence for $p=2q$. In this case, by Lemma \ref{upsidupsi}, we know that $\Upsilon(\Sigma^{2q,q}\un^{iso}) \simeq \Sigma^{\infty}_+y(\Sigma^{2q,q}\un^{iso})$. It follows that:
	$$\Theta^*(\Sigma^{2q,q}\un^{iso}) \simeq \rr^*(\Sigma^{\infty}_+y(\Sigma^{2q,q}\un^{iso}))\simeq \Sigma^{\infty}_+y(\rr(\Sigma^{2q,q}\un^{iso}))\simeq \Sigma^{\infty}_+y(\sph^q)=\sph^{q,q},$$
	which proves the claim.
	\end{proof}

\begin{prop}\label{homgr}
	The functor $\Theta^{*}$ induces an isomorphism of homotopy groups: $$\pi_{p,q}(\un^{iso}) \cong \pi^{\syn}_{p-q,q}$$
	for all $p$, $q \in \Z$, where $\pi^{\syn}_{**}$ denote the $\F$-synthetic homotopy groups of the $\F$-synthetic sphere spectrum $\sph^{0,0}$.
\end{prop}
\begin{proof}
	For $p\geq 2q$, the isomorphism follows from Proposition \ref{rehg} and \cite[Theorem 4.58]{P}. In particular, we deduce that $\Theta^*$ induces an isomorphism:
	$$\Z \cdot \rho \cong  \pi_{-1,-1}(\un^{iso}) \cong \pi^{\syn}_{0,-1} \cong \Z\cdot \lambda,$$
	where we denote by $\lambda$ the parameter of the deformation for $\syn_{\F}$. Therefore, we have an equivalence of $\F$-synthetic spectra:
	$$\Theta^*(\un^{iso}/\rho) \simeq C\lambda,$$
	where $C\lambda \coloneqq \cof(\sph^{0,-1}\xrightarrow{\lambda}\sph^{0,0})$.
	
	For $p<2q$, we proceed by induction on the integer $n=2q-p$. Suppose by induction hypothesis that $\Theta^*$ induces an isomorphism of homotopy groups for all $m < n$. Then, we obtain a morphism of exact sequences:
	$$
	\xymatrix{
		  \pi_{p,q-1}(\un^{iso}) \ar@{->}[r] \ar@{->}[d] &  \pi_{p,q-1}(\un^{iso}/\rho)\ar@{->}[r] \ar@{->}[d]  &  \pi_{p,q}(\un^{iso}) \ar@{->}[r] \ar@{->}[d]	 &  \pi_{p-1,q-1}(\un^{iso}) \ar@{->}[r]\ar@{->}[d]  &\pi_{p-1,q-1}(\un^{iso}/\rho) \ar@{->}[d]  \\
		\pi_{p-q+1,q-1}^{\syn} \ar@{->}[r] &  \pi_{p-q+1,q-1}(C\lambda)\ar@{->}[r]   &  \pi_{p-q,q}^{\syn} \ar@{->}[r] 	 &  \pi_{p-q,q-1}^{\syn} \ar@{->}[r] &\pi_{p-q,q-1}(C\lambda) 
	}
	$$
	where the first and fourth vertical arrows are isomorphisms by induction hypothesis, while the second and fifth are isomorphisms by Theorem \ref{mustlab} and \cite[Lemma 4.56]{P}. It follows that also the middle map is an isomorphism that is what we aimed to show.
	\end{proof}

\begin{thm}\label{upsi}
The functor $\Theta^*$ is an equivalence of stable $\infty$-categories.
\end{thm}
\begin{proof}
	First, we prove that $\Theta^*$ is fully faithful, that is, for all $X$ and $Y$ in $\SH^{iso}_{cell}(\R)$, $\Theta^*$ induces an equivalence:
	$$\mathrm{Map}_{\SH^{iso}_{cell}(\R)}(X,Y) \rightarrow \mathrm{Map}_{\syn_{\F}}(\Theta^*(X),\Theta^*(Y)).$$
	
	We start with the case $X=\Sigma^{0,q}\un^{iso}$ for any $q$. Denote by ${\mathcal C}_q$ the full subcategory of $\SH^{iso}_{cell}(\R)$ spanned by those isotropic cellular spectra $Y$ such that
	$$\mathrm{Map}_{\SH^{iso}_{cell}(\R)}(\Sigma^{0,q}\un^{iso},Y) \rightarrow \mathrm{Map}_{\syn_{\F}}(\Theta^*(\Sigma^{0,q}\un^{iso}),\Theta^*(Y))$$
	is an equivalence. We notice that ${\mathcal C}_q$ contains $\Sigma^{0,q'}\un^{iso}$ for all $q'$ by Lemma \ref{spheres} and Proposition \ref{homgr}. Moreover, it is stable and closed under small coproducts since $\Sigma^{0,q}\un^{iso}$ is compact and $\Theta^*$ is cocontinuous. Hence, ${\mathcal C}_q$ coincides with $\SH^{iso}_{cell}(\R)$ for all $q$.
	
	Now, let $\mathcal{D}$ be the full subcategory of $\SH^{iso}_{cell}(\R)$ spanned by those isotropic cellular spectra $X$ such that
	$$\mathrm{Map}_{\SH^{iso}_{cell}(\R)}(X,Y) \rightarrow \mathrm{Map}_{\syn_{\F}}(\Theta^*(X),\Theta^*(Y))$$
	is an equivalence for all $Y$. Then, $\mathcal{D}$ is stable, closed under small coproducts by the cocontinuity of $\Theta^*$, and contains $\Sigma^{0,q}\un^{iso}$ for all $q$. We conclude that $\mathcal{D}$ coincides with $\SH^{iso}_{cell}(\R)$, which means that $\Theta^*$ is fully faithful.
	
	To finish, we notice that the essential surjectivity of $\Theta^*$ follows from its fully faithfulness and cocontinuity, the fact that $\syn_{\F}$ is generated by the $\F$-synthetic spheres $\sph^{p,q}$ and Lemma \ref{spheres}.
	\end{proof}

\begin{rem}
\normalfont
In Theorem \ref{upsi}, we have decided to show directly that $\Theta^{*}= \rr^* \circ \Upsilon$ is an equivalence. Alternatively, one could follow the approach of \cite{P}, and show first that $\Upsilon$ is an equivalence, and then that $\rr^*$ is as well. This implies in particular that:
$$\SH^{iso}_{cell}(\R) \simeq \sh^{\Sp}_{\Sigma}(\Sp^{fp}_{\mbp^{iso}}) \simeq \syn_{\F}.$$
\end{rem}

\section{Relation to real Artin-Tate motivic spectra}

In this section, we study the relation between real Artin-Tate motivic spectra and real isotropic cellular spectra. We begin by recalling some definitions and notations from \cite{BHS}, where the category of real Artin-Tate motivic spectra is studied in detail.

\begin{dfn}
	\normalfont
Denote by $\SH(\R)^{\mathrm{AT}}$ the category of Artin-Tate $\R$-motivic spectra, that is, the full localizing subcategory of $\SH(\R)$ generated by $\Sigma^{0,q}\un$ and $\Sigma^{0,q}\Sigma^{\infty}_+\spec(\C)$ for all $q \in \Z$.

Let $\un_2$ be the 2-completion of the real motivic sphere spectrum in $\SH(\R)^{\mathrm{AT}}$. Following \cite{BHS}, we denote by $\SH(\R)^{\mathrm{AT}}_{i2}$ the category $\Mod(\SH(\R)^{\mathrm{AT}};\un_2)$.
\end{dfn}

In $\SH(\R)^{\mathrm{AT}}$ there are some notable objects and maps whose definitions we now recall. Denote by $\sph^{\C}$ the cofiber of the map $\Sigma^{\infty}_+\spec(\C) \rightarrow \un$, and by $a$ the induced map $\un \rightarrow \sph^{\C}$. This allows us to define trigraded spheres $\sph^{p,q,w} \coloneqq \Sigma^{p+w,w}\un \wedge (\sph^{\C})^{\wedge q-w}$.

By \cite[Theorem 2.1]{BHS}, there exists a morphism $\ta:\sph^{0,0,-1} \rightarrow \un_2$ that maps to $\tau\in \pi_{0,-1} ({\un_{\C,2}})$ under base change to $\C$. Then, \cite[Theorem 1.4]{BHS} describes $\SH(\R)^{\mathrm{AT}}_{i2}$ as a one-parameter deformation of the $C_2$-equivariant stable homotopy category with parameter $\ta$ and special fiber given by the derived category of Mackey-functor $\MU_*\MU$-comodules.

\begin{thm}\label{tosyn}
The isotropic localization functor $L^{iso}:\SH(\R) \rightarrow \SH^{iso}(\R)$ restricts to a functor $L^{iso}:\SH(\R)^{\mathrm{AT}} \rightarrow \SH^{iso}_{cell}(\R)$ that factors through the category $\SH(\R)^{\mathrm{AT}}[a^{-1}]$ and sends $\sph^{p,q,w}$ to $\Sigma^{p+w,w}\un^{iso}$.
\end{thm}
	\begin{proof}
		Since $\spec(\C) \cong Q_{\rho}$, it follows from Remark \ref{annull} that $L^{iso}(\Sigma^{\infty}_+\spec(\C)) \simeq 0$. Therefore, we immediately see that the essential image of $\SH(\R)^{\mathrm{AT}}$ under $L^{iso}$ is contained in $\SH^{iso}_{cell}(\R)$. Moreover, we notice that $L^{iso}(a):\un^{iso} \rightarrow L^{iso}(\sph^{\C})$ is an equivalence, which implies that $L^{iso}$ factors through $\SH(\R)^{\mathrm{AT}}[a^{-1}]$ and
		$$L^{iso}(\sph^{p,q,w}) =L^{iso}(\Sigma^{p+w,w}\un \wedge (\sph^{\C})^{\wedge q-w})\simeq \Sigma^{p+w,w}\un^{iso} \wedge (\un^{iso})^{\wedge q-w} \simeq \Sigma^{p+w,w}\un^{iso} .$$
		
		This completes the proof.
		\end{proof}
	
	\begin{rem}
		\normalfont
		Denote by $R:\SH^{iso}_{cell}(\R) \rightarrow \SH(\R)^{\mathrm{AT}}$ the right adjoint to the cocontinuous functor $L^{iso}:\SH(\R)^{\mathrm{AT}} \rightarrow \SH^{iso}_{cell}(\R)$ from Theorem \ref{tosyn}. Since $R$ preserves 2-completions in the respective categories, we obtain an induced adjunction:
		$$L^{iso}_2:\SH(\R) ^{\mathrm{AT}}_{2}\leftrightarrows \SH^{iso}_{cell}(\R)_{2}:R$$
		where $L^{iso}_2\coloneqq (-)^{\wedge}_2 \circ L^{iso}$, and $\SH(\R) ^{\mathrm{AT}}_{2}$ and $\SH^{iso}_{cell}(\R)_{2}$ are the subcategories of $2$-complete objects.
		
		Let $\un^{iso}_2$ be the 2-completion of $\un^{iso}$ in $\SH^{iso}_{cell}(\R)$, and let $\SH^{iso}_{cell}(\R)_{i2}\coloneqq \Mod(\SH^{iso}_{cell}(\R);\un^{iso}_2)$. Since $L^{iso}_2(\un_2)\simeq \un^{iso}_2$, we get a symmetric monoidal functor $L^{iso}_{i2}:\SH(\R) ^{\mathrm{AT}}_{i2}\rightarrow \SH^{iso}_{cell}(\R)_{i2}$ that restricts to $L^{iso}_2$ on $2$-complete objects.
	\end{rem}

\begin{prop}\label{syn2comp}
	The functor $L^{iso}_{i2}:\SH(\R) ^{\mathrm{AT}}_{i2}\rightarrow \SH^{iso}_{cell}(\R)_{i2}$ preserves the one-parameter deformation, that is, it sends $\sph^{p,q,w}_2$ to $\Sigma^{p+w,w}\un^{iso}_2$ and $\ta$ to $\rho$ up to a sign.
\end{prop}
\begin{proof}
	Theorem \ref{tosyn} immediately implies that $L^{iso}_{i2}$ sends $\sph^{p,q,w}_2$ to $\Sigma^{p+w,w}\un^{iso}_2$, so we only need to check that $L_{i2}^{iso}(\ta) \simeq \pm\rho$. Since $L^{iso}$ factors through $\SH(\R)^{\mathrm{AT}}[a^{-1}]$, this follows from the relation $\rho=\ta\cdot a$ in $\pi_{0,-1,-1}(\un_2)=[\sph^{0,-1,-1},\un_2]$ proved in \cite[Corollary 3.6]{BHS}, and from the fact that $L^{iso}(a)$ is a unit in $\pi_{0,0}(\un^{iso})\cong \Z$.
	\end{proof}
	
	\begin{rem}
	\normalfont
	Under the identification between real isotropic cellular spectra and $\F$-synthetic spectra proved in the previous section, Theorem \ref{tosyn} provides a cocontinuous functor $L: \SH(\R)^{\mathrm{AT}} \rightarrow \syn_{\F}$ that factors through $\SH(\R)^{\mathrm{AT}}[a^{-1}]$ and sends $\sph^{p,q,w}$ to $\sph^{p,w}$, by Lemma \ref{spheres}. 
	
	By Proposition \ref{syn2comp}, we also obtain a functor:
	$$L^{\wedge}_{i2}:\SH(\R) ^{\mathrm{AT}}_{i2}\rightarrow\syn_{\F,i2}$$
	where $\syn_{\F,i2} \coloneqq \Mod(\syn_{\F};\sph^{\wedge}_2)$.
	
	Since the functor $L^{\wedge}_{i2}$ factors through $\SH(\R) ^{\mathrm{AT}}_{i2}[a^{-1}]$, in the end we get a functor:
	$$\mathrm{Re}_{\F}:\SH(\R) ^{\mathrm{AT}}_{i2}[a^{-1}] \rightarrow \syn_{\F,i2}$$
	which can be identified with the functor constructed in \cite[Proposition 7.6]{BHS}.
	\end{rem}

Recall that there is an isomorphism of Hopf algebroids:
$$(\mathrm{BP}_*, \mathrm{BP}_*\mathrm{BP})\cong (\Z_{(2)}[v_1,v_2,\dots],\Z_{(2)}[v_1,v_2,\dots,t_1,t_2,\dots]).$$ 
	
	\begin{prop}\label{last}
	The functor $L^{iso}:\SH(\R)^{\mathrm{AT}} \rightarrow \SH^{iso}_{cell}(\R)$ induces a morphism of Hopf algebroids:
	$$(\pi_{2*,*}(\mbp),\mbp_{2*,*}(\mbp)) \rightarrow (\pi_{2*,*}(\mbp^{iso}),\mbp^{iso}_{2*,*}(\mbp^{iso}))$$
	that is the quotient map
	$$(\mathrm{BP}_*, \mathrm{BP}_*\mathrm{BP}) \rightarrow (\F, \A_*)$$
	sending $v_i$ to 0 and $t_i$ to $\xi_i$.
	\end{prop}
\begin{proof}
The map of homotopy commutative ring spectra $\mbp \rightarrow \mbp^{iso}$ induces a morphism of Hopf algebroids: $$(\pi_{**}(\mbp),\mbp_{**}(\mbp)) \rightarrow (\pi_{**}(\mbp^{iso}),\mbp^{iso}_{**}(\mbp^{iso})),$$ 
	which in Chow-Novikov degree 0 factors as:
	$$\mbp_{2*,*}(\mbp) \rightarrow \mbp_{2*,*}(\hz) \rightarrow \mbp^{iso}_{2*,*}(\hz^{iso}) \rightarrow \mbp^{iso}_{2*,*}(\mbp^{iso}).$$
	
	It follows from Proposition \ref{hmbpiso}, Proposition \ref{freembp} and Theorem \ref{pimbp} that this composition is the quotient map
	$$(\mathrm{BP}_*, \mathrm{BP}_*\mathrm{BP}) \rightarrow (\F, \A_*)$$ 
	by the ideal $(2,v_1,v_2,\dots)$. This completes the proof.
	\end{proof}

\begin{rem}
\normalfont
Proposition \ref{syn2comp} and Proposition \ref{last} imply in particular that the functor
$$\Mod(\SH(\R)^{\mathrm{AT}}_{i2};C\ta) \rightarrow \Mod(\SH^{iso}_{cell}(\R)_{i2};\un^{iso}/\rho)\simeq \Mod(\syn_{\F,i2};C\lambda)$$ 
is induced by the map of Hopf algebroids $(\mathrm{BP}_*, \mathrm{BP}_*\mathrm{BP}) \rightarrow (\F, \A_*)$ which sends $t_i$ to $\xi_i$. This answers a question posed by Burklund-Hahn-Senger in \cite[Question 7.7]{BHS}.
\end{rem}

\footnotesize{
	
}

\noindent {\scshape Dipartimento di Matematica e Applicazioni, Universit\`a degli Studi di Milano-Bicocca}\\
fabio.tanania@gmail.com

\end{document}